\documentclass{article}

\usepackage{lastpage}

\usepackage{hyperref}
\hypersetup{colorlinks,linkcolor={blue},citecolor={blue},urlcolor={red}}  
\usepackage[square,numbers]{natbib}

\usepackage{fancyhdr,graphicx,amsmath,amssymb}
\usepackage[ruled,vlined]{algorithm2e}
\include{pythonlisting}
\usepackage{amsmath}
\usepackage{graphicx}
\usepackage{geometry}
\usepackage{lscape}
\usepackage{makecell}
\usepackage{subcaption}
\usepackage[export]{adjustbox}
\usepackage{float}

\usepackage{multirow,rotating}

\usepackage{bbm}

\DeclareMathOperator{\diag}{diag}
\DeclareMathOperator{\Score}{\mathcal{L}}

\newcommand{\ignore}[1]{}
\newcommand*\samethanks[1][\value{footnote}]{\footnotemark[#1]}

\newcommand{\revised}[1]{{#1}}

\usepackage{amsthm}

\theoremstyle{remark}
\newtheorem{remark}{Remark}

\theoremstyle{definition}
\newtheorem{definition}{Definition}

\newtheorem{prop}{Proposition}

\newtheorem{cond}{Assumption}

\begin{document}
	\title{Consistent Second-Order Conic Integer Programming  for Learning Bayesian Networks}

\author{Simge K\"u\c{c}\"ukyavuz\thanks{Department of Industrial Engineering and Management Sciences, Northwestern University 
  ({simge@northwestern.edu}).}  \thanks{These authors contributed equally to this work.}
\and Ali Shojaie$^\dagger$  \thanks{Department of Biostatistics, University of Washington
  ({ashojaie@uw.edu}).} 
\and Hasan Manzour\thanks{Department of Industrial and Systems Engineering, University of Washington
  ({hmanzour@uw.edu}).}
\and Linchuan Wei\thanks{Department of Industrial Engineering and Management Sciences, Northwestern University 
  ({LinchuanWei2022@u.northwestern.edu}).}
 \and  Hao-Hsiang Wu \thanks{Department of Management Science, 
		National Yang Ming Chiao Tung University ({hhwu2@nycu.edu.tw}).}
	 }

\ignore{	
	\author{\name Simge K\"u\c{c}\"ukyavuz\thanks{These authors contributed equally to this work.} \email{simge@northwestern.edu} \\
		\addr Department of Industrial Engineering and Management Sciences\\
		Northwestern University\\
		\AND
		\name Ali Shojaie\samethanks \email{ashojaie@uw.edu} \\
		\addr Department of Biostatistics \\ University of Washington\\
		\AND
		\name Hasan Manzour \email{hmanzour@uw.edu} \\
		\addr Department of Industrial and Systems Engineering\\
		University of Washington\\
		\AND
		\name Linchuan Wei  \email{LinchuanWei2022@u.northwestern.edu} \\
		\addr Department of Industrial Engineering and Management Sciences\\
		Northwestern University\\
\AND
		\name Hao-Hsiang Wu  \email{hhwu2@nycu.edu.tw} \\
		\addr  Department of Management Science\\
		National Yang Ming Chiao Tung University\\
	}
}	
	
\maketitle
	
\begin{abstract}
Bayesian Networks (BNs) represent conditional probability relations among a set of random variables (nodes) in the form of a directed acyclic graph (DAG), and have found diverse applications in knowledge discovery. We study the problem of learning the sparse DAG structure of a BN from continuous observational data. The central problem can be modeled as a mixed-integer  program with an objective function composed of a convex quadratic loss function and a regularization penalty subject to linear constraints. The  optimal solution to this mathematical program is known to have desirable statistical properties under certain conditions.  However, the state-of-the-art optimization solvers are not able to obtain provably optimal solutions to the existing mathematical formulations for medium-size problems within reasonable computational times. To address this difficulty, we tackle the problem from both computational and statistical perspectives. 
On the one hand, we propose a concrete early stopping criterion to terminate the branch-and-bound process in order to obtain a near-optimal solution to the mixed-integer program, and establish the consistency of this approximate solution. On the other hand, we improve the existing formulations by replacing the linear ``big-$M$" constraints that represent the relationship between the continuous and binary indicator variables with second-order conic constraints.  Our numerical results demonstrate the effectiveness of the proposed approaches. 

\end{abstract}

\section{Introduction}

A Bayesian network (BN) is a probabilistic graphical model consisting of a labeled directed acyclic graph (DAG) $\mathcal{G} = (V, E)$, in which the vertex set $V = \{V_1, \dots, V_m\}$ corresponds to $m$ random variables, and the edge set $E$ prescribes a decomposition of the joint probability distribution of the random variables based on their parents in $\mathcal{G}$. The edge set $E$ encodes Markov relations on the nodes in the sense that each node is conditionally independent of its non-descendents given its parents. BNs have been used in knowledge discovery \citep{spirtes2000causation, ChenDrtonWang19}, classification \citep{aliferis2010local}, feature selection \citep{gao2015structured}, latent variable discovery \citep{lazic2013structural} and genetics \citep{ott2003finding}. They also play a vital part in causal inference \citep{pearl2009causal}.    

In this paper, we \revised{propose reformulations of the}  mixed-integer quadratic programs (MIQP)  for learning the optimal DAG structure of BNs given $n$ continuous observations from a system of linear structural equation models (SEMs). While there exist exact integer-programming (IP) formulations for learning DAG structure with \emph{discrete} data \citep{cussens2010maximum, cussens2012bayesian, hemmecke2012characteristic, studeny2013polyhedral, bartlett2013advances, JMLR:v17:14-479,oates2016exact, bartlett2017integer, cussens2017polyhedral,cussens2017bayesian}, the development of {tailored} computational tools for learning the optimal DAG structure from \emph{continuous} data has received less attention. In principle, exact methods developed for discrete data can be applied to {continuous} data. However, such methods result in exponentially sized formulations in terms of the number of binary variables. A common practice to circumvent the exponential number of binary variables is to limit the in-degree of each node \citep{cussens2012bayesian,cussens2017bayesian, bartlett2017integer}. But, this may result in sub-optimal solutions. On the contrary, MIQP formulations for learning DAGs corresponding to linear SEMs require a \textit{polynomial} number of binary variables. This is because for BNs with linear SEMs, the score function --- i.e., the penalized negative log-likelihood (PNL) --- can be explicitly written as a function of the coefficients of linear SEMs \citep{shojaie2010penalized, van2013ell, park2017bayesian, manzour2019integer}.  \revised{In contrast to the existing MIQPs \citep{park2017bayesian, manzour2019integer}, our reformulations exploit the convex quadratic objective and the relationship between the continuous and binary variables to improve the strength of the continuous relaxations.}

Continuous BNs with linear SEMs have witnessed a growing interest in the statistics and computer science communities \citep{van2013ell, raskutti2013learning, loh2014high, ghoshal2016information, solus2017consistency}. 
In particular, it has been shown that the solution obtained from solving the  PNL augmented by $\ell_0$ regularization, \revised{which introduces a  penalty  on  the number of non-zero  arc weights in the estimated DAG}, achieves desirable statistical properties \citep{peters2013identifiability, van2013ell, loh2014high}. 
Moreover, if the model is \emph{identifiable} \citep{peters2013identifiability, loh2014high}, \revised{that is when the true causal graph can be identified from the joint distribution,} then such a solution is guaranteed to uncover the true causal DAG when the sample size $n$ is large enough. 
However, given the difficulty of obtaining exact solutions, existing approaches for learning DAGs from linear SEMs have primarily relied on \emph{heuristics}, using techniques such as coordinate descent \citep{fu2013learning, aragam2015concave, han2016estimation} and non-convex continuous optimization \citep{zheng2018dags}. 
Unfortunately, these heuristics are not guaranteed to achieve the desirable properties of the global optimal solution. Moreover, it is difficult to evaluate the statistical properties of a sub-optimal solution with no optimality guarantees \citep{koivisto2012advances}. To bridge this gap, in this paper we develop mathematical formulations for learning optimal BNs from linear SEMs using a PNL objective with $\ell_0$ regularization. By connecting the optimality gap of the mixed-integer program to the statistical properties of the solution, we also establish an \emph{early stopping criterion} under which we can terminate the branch-and-bound procedure and attain a solution which asymptotically recovers the true parameters with high probability. 

Our work is related to recent efforts to develop exact tailored methods for DAG learning from continuous data.\ \cite{xiang2013lasso} show that $A^{\ast}$-lasso algorithm tailored for DAG structure learning from continuous data with $\ell_1$-regularization, \revised{which introduces a penalty on the sum of absolute values of the arc weights,}  is more effective than the previous approaches based on dynamic programming  \citep[e.g.,][]{silander2006simple} that are suitable for both discrete and continuous data.\ \cite{park2017bayesian} develop a mathematical program for  DAG structure learning with $\ell_1$ regularization. \cite{manzour2019integer} improve and extend the formulation by \cite{park2017bayesian} for DAG learning from continuous data with both $\ell_0$ and $\ell_1$ regularizations. The numerical experiments by \cite{manzour2019integer} demonstrate that as the number of nodes grows, their MIQP formulation outperforms $A^{\ast}$-lasso and the existing IP methods; this improvement is both in terms of reducing the IP optimality gap, when the algorithm is stopped due to a time limit, and in terms of computational time, when the instances can be solved to optimality. In light of these recent efforts, the current paper makes important contributions to this problem at the intersection of statistics and optimization.  

\begin{itemize}
	\item The statistical properties of \textit{optimal} PNL with $\ell_0$ regularization have been studied extensively \citep{loh2014high,van2013ell}. However, it is often difficult to obtain an optimal  solution and no results have been established on the statistical properties of approximate solutions. In this paper, we give an early stopping criterion for the branch-and-bound process under which the approximate solution gives consistent estimates of the true coefficients of the linear SEM. Our result leverages the statistical consistency of the PNL estimate with $\ell_0$ regularization \citep{van2013ell, peters2013identifiability} along with the properties of the branch-and-bound method wherein both lower and upper bound values on the objective function are available at each iteration. By connecting these two properties, we obtain a concrete early stopping criterion, as well as a  proof of consistency of the approximate solution. To the best of our knowledge, this result is the first of its kind for DAG learning. 
	\item In spite of recent progress, a key challenge in learning DAGs from linear SEMs is enforcing bounds on arc weights. This is commonly modeled using the standard ``big-$M$ constraint" approach \citep{park2017bayesian, manzour2019integer}. As shown by \cite{manzour2019integer}, this strategy leads to poor continuous relaxations for the problem, which in turn results in slow lower bound improvement in the branch-and-bound tree. In particular, \cite{manzour2019integer} establish that all existing big-$M$ formulations achieve the same continuous relaxation objective function under a mild condition (see Proposition~\ref{Prop2}). To circumvent this issue, we present a mixed-integer second-order cone program (MISOCP), which  gives a tighter continuous relaxation than existing big-$M$ formulations \revised{under certain conditions discussed in detail in Section \ref{deltavalue}}. This formulation can be solved by powerful state-of-the-art optimization packages. Our numerical results show the superior performance of MISOCP compared to the existing big-$M$ formulations in terms of improving the lower bound and reducing the optimality gap. \revised{We also compare our method against the state-of-the-art benchmarks \citep{ChenDrtonWang19,pmlr-v84-ghoshal18a} both for identifiable and non-identifiable instances, and show that our method  provides the best estimation among all methods in most of the  networks, especially for the non-identifiable cases.}
	
\end{itemize}

The rest of the paper is organized as follows. In Section~\ref{Sec: SEMs}, we define the DAG structure learning problem corresponding to linear SEMs, and give a general framework for the problem. In Section~\ref{Cons}, we present our early stopping criterion and establish the asymptotic properties of the  solution obtained under this stopping rule. 
 We review existing mathematical formulations in Section~\ref{Sec: Previous work}, and present our proposed mathematical formulations in Section~\ref{Sec: Math models}. 
Results of comprehensive numerical studies are presented in Section~\ref{Sec: Computational}. We end the paper with a summary in Section~\ref{Sec: Conclusion}.  
 
  \raggedbottom 
  
\section{Problem setup: Penalized DAG estimation with linear SEMs} \label{Sec: SEMs}

Let $\mathcal{M} = (V, E)$ be an undirected and possibly cyclic super-structure graph with node set $V=\{1,2,\dots,m\}$ and edge set $E \subseteq V \times V$; let $\overrightarrow{\mathcal{M}} = (V, \overrightarrow{E})$ be the corresponding bi-directional graph with $\overrightarrow{E} =\{(j,k), (k,j) | (j,k) \in E\}$. We refer to undirected edges as \emph{edges} and directed edges as \emph{arcs}.  

We assume that causal effects of continuous random variables in a DAG $\mathcal{G}_0$ are represented by $m$ linear regressions of the form 
\begin{equation} \label{LSLM}
X_k = \sum_{j \in pa^{\mathcal{G}_0}_k} \beta_{jk} X_j + \epsilon_k, \quad k=1,\dots, m,
\end{equation}
\noindent where $X_k$ is the random variable associated with node $k$, $pa^{\mathcal{G}_0}_k$ represents the parents of node $k$ in $\mathcal{G}_0$, i.e., the set of nodes with arcs pointing to $k$; the latent random variable $\epsilon_k$ denotes the unexplained variation in node $k$; and BN parameter $\beta_{jk}$ specifies the effect of node $j$ on $k$ for $j \in pa^{\mathcal{G}_0}_k$. The above model is known as a linear SEM \citep{pearl2009causal}. 

Let $\mathcal{X}=(\mathcal{X}_1, \dots , \mathcal{X}_m)$ be the $n \times m$ data matrix with $n$ rows representing i.i.d.\ samples from each random variable, and $m$ columns representing random variables $X_1, \ldots, X_m$. 
The linear SEM \eqref{LSLM} can be compactly written in matrix form as $\mathcal{X} = \mathcal{X}{B} + \mathcal{E}$, where ${B} = [\beta] \in \mathbb{R}^{m \times m}$ is a matrix with $\beta_{kk}=0$ for $k=1,\dots,m$, $\beta_{jk}=0$ for all $(j,k) \notin E$, and $\mathcal{E}$ is the $n\times m$ `noise' matrix. Then, $\mathcal{G}(B)$ denotes the directed graph on $m$ nodes such that arc $(j,k)$ appears in $\mathcal{G}(B)$ if and only if $\beta_{jk} \neq 0$. Throughout the paper, we will use $B$ and $\beta$ to denote the matrix of coefficients and its vectorized version. 

A key challenge when estimating DAGs by minimizing the loss function \eqref{eqn:lklhd} is that the true DAG is generally not identifiable from observational data. However, for certain SEM distributions, the true DAG is \emph{identifiable} from observational data; \revised{that is when the true causal graph can be identified from the joint distribution}. Two important examples are linear SEMs with possibly non-Gaussian homoscedastic noise variables \citep{peters2013identifiability}, as well as linear SEMs with unequal noise variances that are known up to a constant \citep{loh2014high}. In these special cases, the true DAG can be identified from observational data, without requiring the (strong) `faithfulness' assumption, which is known to be restrictive in high dimensions \citep{uhler2013geometry, sondhi2019reduced}. Given these important implications, in this paper we focus on learning Bayesian networks corresponding to the above \emph{identifiable} linear SEMs, \revised{i.e., settings where the error variances are either equal, or known up to a constant.}

The negative log likelihood for an identifiable linear SEM \eqref{LSLM} with equal noise variances is proportional to  
\begin{equation}\label{eqn:lklhd}
l(\beta; \mathcal{X}) =n\,\text{tr}\left\{(I-{B})(I-{B})^\top \widehat{\Sigma}\right\};
\end{equation}
here $\widehat{\Sigma} =n^{-1} \mathcal{X}^\top \mathcal{X}$ is the empirical covariance matrix, and $I$ is the identity matrix \citep{shojaie2010penalized, van2013ell}. 

To learn \textit{sparse} DAGs, \citet{van2013ell} propose to augment the negative log likelihood with an $\ell_0$ regularization term. Given a super-structure $\mathcal{M}$, the optimization problem corresponding to this penalized negative log-likelihood (PNL$\mathcal{M}$) is given by \begin{subequations} \label{eq:PNLMform}
	\begin{align}
	\textbf{PNL$\mathcal{M}$} \quad	& \underset{B \in {\mathbb R}^{m \times m}}{\min} \quad  \Score(\beta):= l(\beta; \mathcal{X}) + \lambda_n \|\beta\|_0 \label{Eq: Opt} \\
	\text{s.t.} \, \, & \mathcal{G}(B) \, \, \text{induces a DAG from} \, \overrightarrow{\mathcal{M}}, \label{Eq: DAG const}
	\end{align}
\end{subequations}  
where the tuning parameter $\lambda_n$  controls the degree of the $\ell_0$ regularization 
\revised{
\[
\|\beta\|_0:=\sum_{(j,k)\in \overrightarrow E} \mathbbm{1}(\beta_{jk}),
\]
where $ \mathbbm{1}(\beta_{jk})$ is an indicator function with value one if $\beta_{jk}\ne 0$, and 0 otherwise.} The constraint \eqref{Eq: DAG const} stipulates that the resulting directed subgraph is a DAG induced from $\overrightarrow{\mathcal{M}}$. When $\mathcal{M}$ corresponds to a complete graph, PNL$\mathcal{M}$ reduces to the original PNL of \citet{van2013ell}. 

The choice of $\ell_0$ regularization in \eqref{eq:PNLMform} is deliberate. Although $\ell_1$ regularization has attractive computational and statistical properties in high-dimensional regression \citep{bulmann2011statistics}, many of these advantages disappear in the context of DAG structure learning \citep{fu2013learning, aragam2015concave}. By considering $\ell_0$ regularization, \cite{van2013ell} establish the consistency of PNL under appropriate assumptions. More specifically, for a Gaussian SEM, they show that the estimated DAG has (asymptotically) the same number of edges as the DAG with minimal number of edges (minimal-edge I-MAP), and establish the consistency of PNL for learning sparse DAGs. These results are formally stated in Proposition~\ref{prop:van} in the next section. 

\begin{remark}\label{rem:L2}
A Tikhonov ($\ell_2$) regularization term, $\mu \|\beta\|_2^2$, for a given $\mu > 0$, can also be added to the objective \eqref{Eq: Opt} to obtain more stable solutions \citep{bertsimas2016best}. 
\end{remark}

In our earlier work \citep{manzour2019integer}, we observe that existing mathematical formulations are slow to converge to a provably optimal solution, $\beta^\star$, of \eqref{eq:PNLMform} using the state-of-the-art optimization solvers. Therefore, the solution process  needs to be terminated early to yield a feasible solution, $\hat \beta$ with a positive optimality gap, i.e., a positive difference between the upper bound on  $\Score(\beta^\star)$ provided by $\Score(\hat \beta)$  and a lower bound on $\Score(\beta^\star)$ provided by the best continuous relaxation obtained by the branch-and-bound algorithm upon termination.  
However, statistical properties of such a sub-optimal solution are not well-understood. Therefore, there exists a gap between theory and computation: while the optimal solution has nice statistical properties, the properties of the solutions obtained from approximate computational algorithms are not known. Moreover, due to the non-convex and complex nature of the problem, characterizing the properties of the solutions provided by heuristics is especially challenging. In the next section, we bridge this gap by developing a concrete early stopping criterion and establishing the consistency of the solution obtained using this criterion. 

\section{Early stopping criterion for DAG learning} \label{Cons}
In this section, we establish a sufficient condition for the approximate solution of PNL$\mathcal{M}$, $\hat{\beta}$ to be consistent for the true coefficients, $\beta^{0}$; that is  $\|\beta^{0} - \hat{\beta}\|_2^2 = \mathcal{O}\left(s^0\log(m) / n \right)$, where $s^0$ is the number of arcs in the true DAG, and $x \asymp y$ means that $x$ converges to $y$ asymptotically. 
This result is obtained by leveraging an important property of the branch-and-bound process for integer programming that provides both lower and upper bounds on the objective function $ \Score(\beta^\star)$ upon early stopping, as well as the consistency results of the PNL estimate with $\ell_0$ regularization. Using the insight from this new result, we then propose a concrete stopping criterion for terminating the branch-and-bound process that results in consistent parameter estimates. 

Let $\mathrm{LB}$ and $\mathrm{UB}$, respectively, denote the lower and upper bounds on the optimal objective function value \eqref{Eq: Opt}  obtained from solving \eqref{eq:PNLMform} under an early stopping criterion (i.e., when the obtained solution is not necessarily optimal). We define the difference between the upper and lower bounds as the \emph{absolute} optimality gap: $\mathrm{GAP} = \mathrm{UB} - \mathrm{LB}$. Let $\hat{\mathcal{G}}$ and $ \hat{\beta}$ denote the structure of the DAG and coefficients of the arcs from optimization model \eqref{eq:PNLMform} under the early stopping condition with sample size $n$ and regularization parameter $\lambda_n$. 
Let ${\mathcal{G}^{\star}}$ and $\beta^{\star}$ denote the DAG structure and coefficients of arcs obtained from the optimal solution of \eqref{eq:PNLMform}, and $\mathcal{G}^{0}$ and $\beta^{0}$ denote the true DAG structure and the coefficient of arcs, respectively. 
We denote the number of arcs in $\hat{\mathcal{G}}$, $\mathcal{G}^{0}$, and ${\mathcal{G}^{\star}}$ by $\hat{s}$, $s^0$, and $s^{\star}$, respectively. The score value in \eqref{Eq: Opt} of each solution is denoted by $\Score(\phi)$ where $\phi \in \{\beta^{\star}, \hat{\beta}, \beta^0\}$. 

Next, we present our main result. Our result extends \citeauthor{van2013ell}'s result on consistency of PNL$\mathcal{M}$ for the optimal, but computationally unattainable, estimator, $\beta^{\star}$ to an approximate estimator, $\hat\beta$, obtained from early stopping. In the following (including the statement of our main result, Proposition~\ref{EarlyProp}), we assume that the super-structure $\mathcal{M}$ is known \emph{a priori}. The setting where $\mathcal{M}$ is estimated from data is discussed at the end of the section. We begin by stating the key result from \cite{van2013ell} and the required assumptions. Throughout, we consider a Gaussian linear SEM of the form \eqref{LSLM}. We denote the variance of error terms, $\epsilon_j$, by $\sigma_{jj}^2$ and the true covariance matrix of the set of random variables, $(X_1,\ldots, X_m)$ by the $m \times m$ matrix $\Sigma$. 

\begin{cond}\label{cond:1}
\revised{Suppose $m < c_0 n/\log(n)$ for some constant $c_0 > 0$ and }
for some constant $\sigma_0^2$, it holds that $\max_{j=1,\ldots,m}\sigma_{jj}^2 \leq \sigma_0^2$. Moreover, the smallest \revised{and largest} eigenvalues of $\Sigma$, $\kappa_{\min}(\Sigma)$ \revised{and $\kappa_{\max}(\Sigma)$}, \revised{satisfy
\[
	\left(\frac{c_0}{\log(n)}\right)^{1/2} < \underline{\kappa} \leq \kappa_{\min}(\Sigma) < \kappa_{\max}(\Sigma) \leq \overline{\kappa} < \infty
\] 
for constants $\underline{\kappa}$ and $\overline{\kappa}$. 
}
\end{cond}

\begin{cond}\label{cond:2}
Let, as in \cite{van2013ell}, $\widetilde\Omega(\pi)$ be the precision matrix of the vector of noise variables for an SEM given permutation $\pi$ of nodes. Denoting the diagonal entries of this matrix by $\tilde \omega_{jj}$, there exists a constant $\omega_0 > 0$ such that if $\widetilde\Omega(\pi)$ is not a multiple of the identity matrix, then 
\[
	m^{-1} \sum_{j=1}^m\left( (\tilde\omega_{jj})^2 -1 \right)^2 > 1/ \omega_0. 
\]
\end{cond}

\begin{prop} (Theorem 5.1 in \cite{van2013ell}) \label{prop:van}
Suppose Assumptions~\ref{cond:1} and \ref{cond:2} hold and let $\alpha_0:= \min\{\frac{4}{m}, 0.05\}$. Then for an $\ell_0$ regularization parameter  $\lambda \asymp \log(m)/n$, it holds with probability at least $1-\alpha_0$ that 
\[
	\|\beta^{\star}-\beta^{0}\|_2^2 + \lambda s^{\star} = \mathcal{O}\left(\lambda s^0\right).
\]
\end{prop} 
Here, $\lambda=\lambda_n/n$, because the loss function \eqref{eqn:lklhd} is that of \cite{van2013ell} scaled by the sample size $n$.
\revised{The next result establishes the consistency of the approximate estimator, $\hat\beta$, obtained using our proposed early stopping strategy.}

\begin{prop} \label{EarlyProp}
	Suppose Assumptions~\ref{cond:1} and \ref{cond:2} hold and let $\alpha_{0} = \min \{\frac{4}{m}, 0.05\}$ and $\lambda \asymp \log(m) / n$. 
	Then, the estimator $\hat\beta$ obtained from early stopping of the  branch-and-bound process such that $\mathrm{GAP} \asymp \mathcal{O}\left(n\lambda s^{0}\right) = \mathcal{O}\left(\log(m) s^0\right)$ 
satisfies 
\begin{equation*}
\left\|\hat\beta - \beta^{0}\right\|_2^2 \asymp \mathcal{O}\left(\frac{\log(m)}{n}s^0\right)
\end{equation*}
with probability \revised{$(1- \alpha_0)$}.
\end{prop}

\begin{proof}
First, by the triangle inequality \revised{and the fact that $2ab \leq a^2 + b^2, \forall a,b \in 
\mathbb{R}$,} 
\revised{
\begin{align}\label{eqn:newbnd1}
	\nonumber \| \hat\beta - \beta^0 \|_2^2 &\leq 
	\left(\| \hat\beta - \beta^\star \|_2 + \| \beta^\star - \beta^0 \|_2\right)^2 \\ 
	\nonumber &=  \| \hat\beta - \beta^\star \|^2_2 + \| \beta^\star - \beta^0 \|^2_2 + 2 \| \hat\beta - \beta^\star \|_2 \| \beta^\star - \beta^0 \|_2 \\ 
	&\leq 2\| \hat\beta - \beta^\star \|_2^2 + 2\| \beta^\star - \beta^ 0 \|_2^2.
\end{align}
}

\revised{Recall that $\beta$ denotes the vectorized coefficient matrix $B$. Then, in a slight abuse of notation, we denote by $\mathcal{X}$ both the vectorized and a block diagonal version of $\mathcal{X}$ and by $\mathcal{E}$ a vectorized version of error $\mathcal{E}$}. Then $\ell(\beta;\mathcal{X})$ can be written as $\ell(\beta;\mathcal{X}) = \| \mathcal{X} - \mathcal{X}\beta \|_2^2$ (see Eq.~\ref{LSLM}). Then, we can write a Taylor series expansion of $\ell\left(\hat\beta;\mathcal{X}\right)$ around $\ell\left(\beta^\star;\mathcal{X}\right)$ to get
\begin{align}\label{eqn:ls2lklhd}
	\| \mathcal{X}(\hat\beta - \beta^\star) \|_2^2 
	= \ell(\hat\beta;\mathcal{X}) - \ell(\beta^\star;\mathcal{X}) - 2(\hat\beta - \beta^\star)^\top \mathcal{X}^\top \mathcal{X} (\beta^\star - \beta^0) + 
	2(\hat\beta - \beta^\star)^\top \mathcal{X}^\top \mathcal{E}.
\end{align}

\revised{But, Proposition~2.1 in \cite{vershynin2012covmat} states that for every $0<\xi <1$,
\begin{equation}\label{eqn:vershynin}
	\left\| \widehat\Sigma - \Sigma \right\|_2 \leq \left(\frac{m}{n}\right)^{1/2}, 
\end{equation}
with probability $1-\xi$. 
Thus, letting $\xi = \alpha_0$, \eqref{eqn:vershynin} holds in our setting with probability $1 - \alpha_0$. 

Since, by Assumption~\ref{cond:1}, $\kappa_{\min}(\Sigma) \geq \underline{\kappa} > \frac{c_0}{\log(n)}$, by Weyl's theorem we have
\begin{equation*}
	\kappa_{\min}(\mathcal{X}^\top \mathcal{X}) = n \kappa_{\min}\left(\widehat\Sigma\right) > n \underline{\kappa} - (mn)^{1/2} > n \underline{\kappa} - n\left(\frac{c_0}{\log(n)}\right)^{1/2} = n\left(\underline{\kappa} - \left(\frac{c_0}{\log(n)}\right)^{1/2}\right),
\end{equation*}
which means that $\kappa_{\min}(\mathcal{X}^\top \mathcal{X}) > 0$ with probability $1 - \alpha_0$. 
}

\revised{
Denoting $c'_n \equiv \left(\underline{\kappa} - \left(\frac{c_0}{\log(n)}\right)^{1/2}\right)^{-1}$, for large enough $n$, we have that, with probability $1 - \alpha_0$,
\begin{equation}\label{eqn:betalower}
	\| \hat\beta - \beta^\star \|_2^2 \leq n^{-1}c'_n \| \mathcal{X}(\hat\beta - \beta^\star) \|_2^2.
\end{equation}
}
\revised{
Combining \eqref{eqn:ls2lklhd} and \eqref{eqn:betalower}, and using triangle inequality again, we get
}
\begin{align}\label{eqn:upperbnd}
	\| \hat\beta - & \beta^\star \|_2^2 \leq n^{-1}c'_n \| \mathcal{X}(\hat\beta - \beta^\star) \|_2^2 \\ \nonumber	
	 =& \revised{n^{-1}c'_n \left( \ell(\hat\beta;\mathcal{X}) - \ell(\beta^\star;\mathcal{X}) - 2(\hat\beta - \beta^\star)^\top \mathcal{X}^\top \mathcal{X} (\beta^\star - \beta^0) + 2(\hat\beta - \beta^\star)^\top \mathcal{X}^\top \mathcal{E} \right)} \\ \nonumber
	\leq&  \revised{n^{-1}c'_n \left| \ell(\hat\beta;\mathcal{X}) - \ell(\beta^\star;\mathcal{X}) - 2(\hat\beta - \beta^\star)^\top \mathcal{X}^\top \mathcal{X} (\beta^\star - \beta^0) + 2(\hat\beta - \beta^\star)^\top \mathcal{X}^\top \mathcal{E} \right|} \\ \nonumber
	\leq& n^{-1}c'_n \left| \ell(\hat\beta;\mathcal{X}) - \ell(\beta^\star;\mathcal{X}) \right| + 2 n^{-1}c'_n (\hat\beta - \beta^\star)^\top \mathcal{X}^\top \mathcal{X} (\beta^\star - \beta^0) + 2 n^{-1}c'_n \left|(\hat\beta - \beta^\star)^\top \mathcal{X}^\top \mathcal{E} \right| \\ \nonumber
	\leq& n^{-1}c'_n\left| \ell(\hat\beta;\mathcal{X}) - \ell(\beta^\star;\mathcal{X}) \right| + 2n^{-1}c'_n \kappa_{\max}(\mathcal{X}^\top \mathcal{X}) \| \hat\beta - \beta^\star \|_2 \| \beta^\star - \beta^0 \|_2 + 
	2n^{-1}c'_n \| \hat\beta - \beta^\star \|_2 \| \mathcal{X}^\top\mathcal{E} \|_2,
\end{align}
where, as before, $\kappa_{\max}$ denotes the maximum eigenvalue of the matrix. 

\revised{
Using a similar argument as the one used above for the minimum eigenvalue of $\mathcal{X}^\top \mathcal{X}$, by \eqref{eqn:vershynin} we have that, with probability $1-\alpha_0$,
\[
	\kappa_{\max}(\mathcal{X}^\top \mathcal{X}) = n \kappa_{\max}\left(\widehat\Sigma\right) \leq n \kappa_{\max}(\Sigma) + n\left(\frac{c_0}{\log(n)}\right)^{1/2} \leq  n \left(\overline\kappa + \left(\frac{c_0}{\log(n)}\right)^{1/2}\right).
\]
Plugging the above bound into \eqref{eqn:upperbnd} we get
\begin{align}\label{eqn:upperbnd2}
	\| \hat\beta - \beta^\star \|_2^2 & \leq
	n^{-1}c'_n\left| \ell(\hat\beta;\mathcal{X}) - \ell(\beta^\star;\mathcal{X}) \right| \\ \nonumber
	&+ 2 c'_n \left(\overline\kappa + \left(\frac{c_0}{\log(n)}\right)^{1/2}\right) \| \hat\beta - \beta^\star \|_2 \| \beta^\star - \beta^0 \|_2 + 
	2n^{-1}c'_n \| \hat\beta - \beta^\star \|_2 \| \mathcal{X}^\top\mathcal{E} \|_2.
\end{align}
}

Now, let $Z = \left\| \hat\beta - \beta^\star \right\|_2$, 
$
\Pi = 2 c'_n \left[ \left(\overline\kappa + \left(\frac{c_0}{\log(n)}\right)^{1/2}\right) \| \beta^\star - \beta^0 \|_2 + n^{-1}\| \mathcal{X}^\top\mathcal{E} \|_2 \right],
$
and 
$
\Gamma = n^{-1}c'_n\left| \ell(\hat\beta;\mathcal{X}) - \ell(\beta^\star;\mathcal{X}) \right|.
$
Then, the inequality in \eqref{eqn:upperbnd2} can be written as $Z^2 \leq \Pi Z + \Gamma$. \revised{Solving for $Z$ and noting that $Z$, $\Gamma$ and $\Pi$ are non-negative, in order to have $Z^2 \leq \Pi Z + \Gamma$, we must have} $Z \leq \left(\Pi + \sqrt{\Pi^2 + 4\Gamma} \,\right) / 2$. 

Next, let $\mathcal{T}$ be the event under which $\Pi = o(1)$. Then, on this set, we have 
\revised{$Z \leq \left(o(1) + \sqrt{o(1) + 4\Gamma} \,\right) / 2$, or, 
$Z^2 \leq \Gamma + o(1)$}; that is
\begin{equation}\label{eqn:newbnd2}
	\left\| \hat\beta - \beta^\star \right\|_2^2 \leq n^{-1}c'_n\left| \ell(\hat\beta;\mathcal{X}) - \ell(\beta^\star;\mathcal{X}) \right| + o(1). 
\end{equation}
Plugging \eqref{eqn:newbnd2} into \eqref{eqn:newbnd1}, on the set $\mathcal{T}$ we have
\begin{align}\label{eqn:newbnd3}
	\| \hat\beta - \beta^0 \|_2^2 & \leq 
		2 n^{-1} c'_n\left| \ell(\hat\beta;\mathcal{X}) - \ell(\beta^\star;\mathcal{X}) \right| + 2\| \beta^\star - \beta^ 0 \|_2^2 + o(1) \\
		& \revised{= 2 n^{-1} c'_n 
		\left| \ell(\hat\beta; \mathcal{X}) - \ell(\beta^{\star}; \mathcal{X}) + (\lambda \hat{s} - \lambda s^\star) - (\lambda \hat{s} - \lambda s^\star) \right| + 2\| \beta^{\star} - \beta^{0} \|_2^2 + o(1)} \nonumber \\
		& \leq 2 n^{-1} c'_n 
		\left| \ell(\hat\beta; \mathcal{X}) - \ell(\beta^{\star}; \mathcal{X}) + \lambda \hat{s} - \lambda s^\star \right| + 2 n^{-1} c'_n \left| \lambda \hat{s} - \lambda s^\star \right|  + 
2\| \beta^{\star} - \beta^{0} \|_2^2 + o(1). \nonumber
\end{align}

Then, using the fact that 
$ \ell(\hat\beta; \mathcal{X}) - \ell(\beta^{\star}; \mathcal{X}) + \lambda \hat{s} - \lambda s^\star = \Score(\hat{\beta}) - \Score(\beta^{\star}) \leq \mathrm{GAP}$
we can write \eqref{eqn:newbnd3} as
\begin{align}\label{eqn:newbnd4}
	\| \hat\beta - \beta^0 \|_2^2 \, & \revised{\leq 
		2 n^{-1} c'_n 
		\left| \Score(\hat\beta; \mathcal{X}) - \Score(\beta^{\star}; \mathcal{X}) \right| + 2 n^{-1} c'_n \lambda s^\star  + 2\| \beta^{\star} - \beta^{0} \|_2^2 + o(1)}, \nonumber \\
		& \leq 2 n^{-1} c'_n \mathrm{GAP} + 2 \| \beta^{\star} - \beta^{0} \|_2^2 +  2 n^{-1} c'_n \lambda s^{\star} + o(1),
\end{align} 
where, in the first inequality, we also use the fact that the penalized estimation procedure chooses at most $n$ nonzero entries. Hence, \revised{since $m < n$ by Assumption~\ref{cond:1}},
\revised{
\[
2n^{-1}c'_n\left| \lambda \hat{s} - \lambda s^\star \right| \leq 2n^{-1}c'_n \lambda s^\star + 2c'_n \lambda = 2n^{-1}c'_n \lambda s^\star + \mathcal{O}\left(\frac{\log(m)}{n}\right) = 2n^{-1}c'_n \lambda s^\star + o(1).
\]
}
Now, by Proposition~\ref{prop:van}, we know that with probability at least $1 - \alpha_0$, $\| \beta^{\star} - \beta^{0} \|_2^2 = \mathcal{O}\left(s^0\log(m)/n\right)$, and $ \lambda s^{\star} = \mathcal{O}\left(s^0\log(m)/n\right)$. 
Moreover, using Gaussian concentration inequalities for $n^{-1}\| \mathcal{X}^\top\mathcal{E} \|_2$
\citep[e.g., Lemma~6.2 in][]{bulmann2011statistics},
the probability of the set $\mathcal{T}$ is lower bounded by the probability that $\| \beta^{\star} - \beta^{0} \|_2^2 = \mathcal{O}\left(s^0\log(m)/n\right)$, which is $1- \alpha_0$. 
Thus, if we stop the branch-and-bound algorithm when 
$$
\mathrm{GAP} = \mathcal{O}\left(n\lambda s^{0}\right) = \mathcal{O}\left(\log(m) s^{0}\right),
$$
then \revised{the first two terms in \eqref{eqn:newbnd4} would both be of order $\mathcal{O}\left( s^0 \log(m) / n \right)$, while the third term, $2n^{-1}c'_n \lambda s^\star$ would be of a smaller order (by an $n^{-1}$ factor)}.
This guarantees that, with probability at least $(1- \alpha_0)$, $\|\hat{\beta}-\beta^0\|_2^2 = \mathcal{O}\left( s^0 \log(m) / n \right)$, as desired. 
\end{proof}

Proposition~\ref{EarlyProp} suggests that the branch-and-bound algorithm can be stopped by setting a threshold $c^{\star} n \lambda s^{0}$ on the value of $\mathrm{GAP} = | \mathrm{UB} - \mathrm{LB} |$ for a constant $c^{\star} > 0$, say $c^{\star}=1$. Such a solution will then achieve the same desirable statistical properties \revised{(in terms of parameter consistency)} as the optimal solution $\beta^{\star}$. 
While $\lambda$ can be chosen data-adaptively (as discussed in Section~\ref{Sec: Computational}), both of these choices depend on the value of $s^0$, which is not known in practice. However, one can find an upper bound for $s^0$ based on the number of edges in the super-structure $\mathcal{M}$. In particular, if $\mathcal{M}$ is the moral graph \citep{pearl2009causal} with $s_m$ edges, then $s^0 \leq s_m$. 
\revised{
However, while in many applications $s_m \asymp s^0$, this is not always guaranteed. Thus, to ensure consistent estimation when replacing $s^0$ with $s_m$ and setting $c^{\star} =1$ in practice, we use the more conservative threshold of $\lambda s^0 \asymp s^{0} \log(m)/n$. With this choice the first and third terms in \eqref{eqn:newbnd4} would be of the same (vanishing) order, and the consistency rate would be driven by the convergence rate of $\| \beta^{\star} - \beta^{0} \|_2^2$. We investigate the performance of this choice in Section~\ref{sec:compearly}.}

The above results, including the specific choice of early stopping criterion, are also valid if the super-structure $\mathcal{M}$ corresponding to the moral graph is not known \emph{a priori}. That is because the moral graph can be consistently estimated from data using recent developments in graphical modeling; see \citet{drton2017structure} for a review of the literature. While  some of the existing algorithms based on $\ell_1$-penalty require an additional \emph{irrepresentability} condition \citep{meinshausen2006high, saegusa2016joint}, this assumption can be relaxed by using instead an adaptive lasso penalty or by thresholding the initial lasso estimates \citep{bulmann2011statistics}.

In light of Proposition \ref{EarlyProp}, it is of great interest to develop algorithms that converge to a solution with a small optimality gap expeditiously. To achieve this, one approach is to obtain better lower bounds using the branch-and-bound process from strong mathematical formulations for  \eqref{eq:PNLMform}. To this end, we next review existing formulations of  \eqref{eq:PNLMform}.

\section{Existing Formulations of DAG Learning with Linear SEMs} \label{Sec: Previous work}
In this section, we review known mathematical formulations for DAG learning with linear SEMs. We first outline the necessary notation below. \\ \\
\noindent \textbf{Index Sets}\\
$V = \{1,2,\dots,m\}$:  index set of random variables;\\
$\mathcal{D}= \{1,2,\dots,n\}$: index set of samples. \vspace{0.1in}\\
\noindent \textbf{Input} \\
$\mathcal{M}=(V,E)$: an undirected super-structure graph (e.g., the moral graph);\\
$\overrightarrow{\mathcal{M}}=(V,\overrightarrow{E})$: the bi-directional graph corresponding to the undirected graph $\mathcal{M}$; \\
$\mathcal{X} = (\mathcal{X}_1, \dots, \mathcal{X}_m)$, where $\mathcal{X}_v = (x_{1v}, x_{2v}, \dots, x_{nv})^{\top}$ and $x_{dv}$ denotes $d$th sample ($d \in \mathcal{D}$) of random variable $X_v$; note $\mathcal{X} \in \mathbb{R}^{n \times m}$; \\
$\lambda_n:$ tuning parameter (penalty coefficient for $\ell_0$ regularization).\\

\noindent \textbf{Continuous optimization variables} \\
$\beta_{jk}$: weight of arc $(j, k)$ representing the regression coefficients $\forall (j,k) \in \overrightarrow{E}$.\\

\noindent\textbf{Binary optimization variables} \\
$z_{jk}=1 \, \, \text{if arc} \, \,  (j, k) \, \text{exists in a DAG}; \text{otherwise} \, 0, \, \forall (j,k) \in \overrightarrow{E}$, \\
$g_{jk}=1 \, \, \text{if} \, \, \beta_{jk} \neq 0; \, \text{otherwise}  \, \, 0, \, \forall (j,k) \in \overrightarrow{E}$. 

Let $F(\beta, g)= \sum_{k\in V}\sum_{d\in \mathcal{D}} \Big(x_{dk}-\sum_{(j,k) \in \overrightarrow{E}} \beta_{jk}x_{dj}\Big)^2 + \lambda_n\sum_{(j,k) \in \overrightarrow{E}}  g_{jk}$.
The PNL$\mathcal{M}$ problem can be cast as the following optimization model:  
\begin{subequations}
	\begin{alignat}{3}
	\label{CP-obj} \quad  \quad   \underset{ \revised{B \in {\mathbb R}^{m\times m}, g\in \{0,1\}^{|\overrightarrow{E}|}}}{\min}\quad & \, F(\beta, g), \\
		&\label{CP-con2}\mathcal{G}(B) \, \, \text{induces a DAG from} \, {\overrightarrow{\mathcal{M}}},  \\
	& \label{CP-con1} \beta_{jk}(1-g_{jk})=0,  \quad && \forall (j,k) \in \overrightarrow{E,}\\
	&\label{CP-con3}  g_{jk} \in \{0,1\},\quad && \forall (j,k) \in \overrightarrow{E}. 
	\end{alignat} \label{MIQP1}
\end{subequations}

\noindent The objective function \eqref{CP-obj} is an expanded version of $\mathcal L(\beta)$ in PNL$\mathcal{M}$, where we use the indicator variable $g_{jk}$ to encode the $\ell_0$ regularization. The constraints in \eqref{CP-con2} rule out cycles. The constraints in \eqref{CP-con1} are non-linear and stipulate that $\beta_{jk} \neq 0$ only if $g_{jk}=1$.

There are two sources of difficulty in solving \eqref{CP-obj}-\eqref{CP-con3}: (i) the acyclic nature of DAG imposed by the combinatorial constraints in \eqref{CP-con2}; (ii) the set of \textit{nonlinear} constraints in \eqref{CP-con1}, which stipulates that $\beta_{jk} \neq 0$ only if there exists an arc $(j,k)$ in $\mathcal{G}(B)$.\ In Section \ref{lit1}, we discuss related studies to address the former, whereas in Section \ref{lit2} we present relevant literature for the latter. 

\subsection{Linear encodings of the acyclicity constraints \eqref{CP-con2}} \label{lit1}
There are several ways to ensure that the estimated graph does not contain any cycles. The first approach is to add a constraint for each cycle in the graph, so that at least one arc in this cycle must not exist in $\mathcal G(B)$.\ A \textit{cutting plane} (CP) method is used to solve such a formulation which may require generating an exponential number of constraints. 
Another way to rule out cycles is by imposing constraints such that the nodes follow a topological order \citep{park2017bayesian}. A topological ordering is a unique ordering of the nodes of a graph from 1 to $m$ such that the graph contains an arc $(j,k)$ if node $j$ appears before node $k$ in the order.  We refer to this formulation as \textit{topological ordering} (TO). \revised{The TO formulation has $\mathcal{O}(m^2)$ variables and $\mathcal{O}(|\overrightarrow{E}|)$ constraints.}  We give these formulations in the Appendix, for completeness. 

The \textit{layered network} (LN) formulation  \revised{for learning from continuous data} proposed by \cite{manzour2019integer} \revised{is shown to perform better, empirically, than} the TO formulation \revised{because it reduces the number of binary variables and is proven to obtain the same continuous relaxation bounds. Therefore, smaller quadratic programs are solved that provide the same relaxation bounds as larger quadratic programs}. \revised{This formulation is closely related to the generation number approach proposed in \cite{cussens2010maximum} for discrete data.} \revised{In this paper, we focus on the LN formulation and refer the reader to the Appendix and  \cite{manzour2019integer} for comparisons of these formulations \revised{and their sizes} in detail. Next, we give the LN encoding of the acyclicity constraints.}
 Define decision variables $z_{jk} \in \{0,1\}$ for all $(j,k) \in \overrightarrow{E}$, where the variable $z_{jk}$ takes value 1 if there is an arc $(j,k)$ in the network
\begin{subequations}\label{LN}
	\begin{alignat}{3} 
		\label{LN-con3} \textbf{LN} \quad  & g_{jk}  \leq z_{jk} \quad&& \forall (j,k) \in \overrightarrow{E}, \\
	\label{LN-con4} & z_{jk} - (m-1) z_{kj} \leq  \psi_k - \psi_j \quad &&\forall (j,k) \in \overrightarrow{E}.
	\end{alignat} 
\end{subequations}
 Let $\psi_k$ be the \textit{layer value} for node $k$. The set of constraints in \eqref{LN-con4} ensures that if the layer of node $j$ appears before that of node $k$ (i.e., there is a direct path from node $j$ to node $k$), then $\psi_k \geq \psi_j + 1$. This rules out any cycles. 

Constraint \eqref{LN-con4} \revised{written for $(k,i)\in  \overrightarrow{E}$}  imposes that if $z_{ij} = 1$ and $z_{jk} = 1$, \revised{i.e., if $\psi_k>\psi_i$}, then $z_{ik} = 1$\revised{, even if $\beta_{ik}=g_{ik}=0$ in an optimal solution}. Thus, additional binary vector $z$ along with the set of constraints in \eqref{LN-con3} is needed to correctly encode the $\ell_0$ regularization. \revised{The LN formulation has $\mathcal{O}(|\overrightarrow{E}|)$ variables and constraints. Note that $|\overrightarrow{E}|$ is much smaller than $m^2$ for sparse skeleton/moral graphs.}

\subsection{Linear encodings of the non-convex constraints \eqref{CP-con1}} \label{lit2}
The nonconvexity of the set of constraints in \eqref{CP-con1} causes challenges in obtaining provably optimal solutions with existing optimization software. Therefore, we consider convex representations of this set of constraints. 
First, we present a linear encoding of the constraints in  \eqref{CP-con1}. Although the existing \revised{compact (i.e., polynomial sized) TO and LN} formulations discussed in Section \ref{lit1} differ in their approach to ruling out cycles, one  commonality among them is that they replace the non-linear constraint \eqref{CP-con1} by so called \emph{big-$M$ constraints} given by
\begin{equation}\label{eq:bigM}
-M g_{jk} \leq \beta_{jk} \leq M g_{jk}, \forall (j,k) \in \overrightarrow{E}, 
\end{equation}
for a large enough $M$. Unfortunately, these big-$M$ constraints \eqref{eq:bigM} are poor approximations of \eqref{CP-con1}, especially in this problem, because no natural and tight value for $M$ exists. Although a few techniques have been proposed for obtaining the big-$M$ parameter for sparse regression problem \cite[e.g.,][]{bertsimas2016best,bertsimas2017sparse,gomez2018mixed,Park2020}, the resulting parameters are often too large in practice. Further, finding a tight big-$M$ parameter itself is a difficult problem to solve for DAG structure learning.

Consider \eqref{CP-obj}-\eqref{CP-con3} by \revised{replacing}  \eqref{CP-con1} \revised{with} the linear big-$M$ constraints \eqref{eq:bigM} and writing the objective function in a matrix form. We denote the resulting formulation, which has a convex quadratic objective and linear constraints, by the following {MIQP}.
\begin{subequations}\label{eq:LNform}
	\begin{alignat}{3}
	\quad  \label{L-obj} \textbf{MIQP}\quad \underset{\revised{B \in {\mathbb R}^{m\times m}, g\in \{0,1\}^{|\overrightarrow{E}|}}}{\min} & \quad \text{tr}\left[(I- B)(I-B)^{\top}\mathcal{X}^{\top}\mathcal{X}\right] + \lambda_n \sum_{(j,k) \in \overrightarrow{E}}  g_{jk}\\
    &  \eqref{CP-con2}, \eqref{eq:bigM}  \label{LN-con1} \\
	& \label{LN-con5} g_{jk} \in \{0,1\}  \quad \forall (j,k) \in \overrightarrow{E}. 
	\end{alignat}
\end{subequations}  

Depending on which types of constraints are used in lieu of \eqref{CP-con2}, as explained in Section \ref{lit1}, {MIQP} \eqref{eq:LNform} results in three different formulations: {MIQP+CP}, which uses \eqref{CE},  {MIQP+TO}, which uses \eqref{TO}, and {MIQP+LN},  which uses \eqref{LN}, respectively.

To discuss the challenges of the big-$M$ approach, we give a definition followed by two propositions.   

\begin{definition}\label{def:2}
	A formulation $A$ is said to be \emph{stronger} than formulation $B$ if $\mathcal{R}(A) \subset \mathcal{R} (B)$ where $\mathcal{R}(A)$ and $\mathcal{R}(B)$ correspond to the feasible regions of continuous relaxations of $A$ and $B$, respectively. 
\end{definition}

\begin{prop}{(Proposition 3 in \cite{manzour2019integer})} 
	{\it The {MIQP+TO} and {MIQP+CP} formulations are stronger than the {MIQP+LN} formulation.} \label{Prop:strong}
\end{prop} 

\revised{As a consequence of Definition \ref{def:2}, the optimal objective function value of  the continuous relaxation of the  stronger formulation  provides a lower bound on the true optimal objective function of the MIQP that is greater than or equal to the optimal objective function value of the continuous relaxation of the weaker formulation due to the smaller set of feasible solutions. However, next proposition shows that, perhaps surprisingly, the continuous relaxations of {MIQP+TO} and {MIQP+CP} formulations, while stronger according to Definition \ref{def:2}, give the same optimal objective function value (and the same lower bound on the true optimal objective).}

\begin{prop}{(Proposition 5 in \cite{manzour2019integer}) \label{Prop2}}	
	\it{Let $\beta^{\star}_{jk}$ denote the optimal coefficient associated with an arc $(j,k) \in \overrightarrow{E}$ from problem \eqref{eq:PNLMform}.\ For the same variable branching in the branch-and-bound process, the continuous relaxations of the {MIQP+LN} formulation for $\ell_0$ \revised{regularization} attain the same optimal objective function value as {MIQP+TO} and {MIQP+CP}, if $M \geq 2  \underset{(j,k) \in \overrightarrow{E}}{\max} \,  |\beta^{\star}_{jk}|$.} \label{Prop5: BB}
\end{prop}

Proposition \ref{Prop:strong} implies that the  {MIQP+TO} and {MIQP+CP}   formulations are stronger than the {MIQP+LN} formulation. Nonetheless, Proposition \ref{Prop5: BB} establishes that for sufficiently large values of $M$, stronger formulations \revised{for the DAG learning problem} attain the same continuous relaxation objective function value as the weaker formulation throughout the branch-and-bound tree.  The optimal solution to the continuous relaxation of MIQP formulations of DAG structure learning may not be at an extreme point of the convex hull of feasible points. Hence, stronger formulations do not necessarily ensure better lower bounds \revised{for certain formulations of this problem involving the \emph{nonlinear} objective}. This is in contrast to a mixed-integer program (MIP) with \emph{linear} objective, whose continuous relaxation is a linear program (LP). In that case, there exists an optimal solution that is an extreme point of the corresponding feasible set.  As a result, a better lower bound can be obtained from a stronger formulation that better approximates the convex hull of \revised{the set of solutions to} a mixed-integer linear program; this generally leads to faster convergence. A prime example is the traveling salesman problem (TSP), for which stronger formulations attain better computational performance \citep{oncan2009comparative}. In contrast, the numerical results by \cite{manzour2019integer} \revised{empirically} show that {MIQP+LN} has better computational performance because it is a compact formulation with the fewest constraints and the same continuous relaxation bounds.

Our next result, which is adapted from \cite{dong2015regularization} to the DAG structure learning problem, shows that the continuous relaxation of {MIQP} is equivalent to  the optimal solution to the problem where the $\ell_0$-regularization term is replaced with an $\ell_1$-regularization term (i.e., $\|\beta\|_1=\sum_{(j,k) \in \overrightarrow{E}}|\beta_{jk}|$) with a particular choice of the $\ell_1$ penalty. This motivates us to consider tighter continuous relaxation for MIQP. 
Let $(\beta^R, g^R)$ be an optimal solution to the continuous relaxation of {MIQP}.

\begin{prop}\label{prop:Prop3}
	For $M \geq 2 \underset{(j,k) \in \overrightarrow{E}}{\max} \, |\beta^R_{jk}|$, a continuous relaxation of {MIQP} \revised{\eqref{eq:LNform}}, where the binary variables are relaxed, is equivalent to the problem where the $\ell_0$ regularization term is replaced with an $\ell_1$-regularization term with penalty parameter $\tilde{\lambda}=\frac{\lambda_n}{M}$. \label{L1}
\end{prop} 
\begin{proof} 
For $M \geq 2 \underset{(j,k) \in \overrightarrow{E}}{\max} \, |\beta^R_{jk}|$, the value $g^R_{jk}$ is $\frac{\beta^R_{jk}}{M}$ in an optimal solution to the continuous relaxation of {MIQP} \eqref{eq:LNform}. Otherwise, we can reduce the value of the decision variable $g^R$ without violating any constraints while reducing the objective function. Note that since $M \geq 2 \underset{(j,k) \in \overrightarrow{E}}{\max} \, |\beta^R_{jk}|$, we have $\frac{\beta_{jk}^R}{M} \leq 1, \, \forall (j,k) \in \overrightarrow{E}$. To show that the set of constraints in \eqref{CP-con2} is satisfied, we consider the set of CP constraints. In this case, the set of constraints \eqref{CP-con2} holds, i.e., $\sum_{(j,k ) \in \, \mathcal{C}_A} \frac{\beta^{R}_{jk}}{M} \leq |\mathcal{C}_A|-1, \quad  \forall \mathcal{C}_A \in \mathcal{C}$, because $M \geq 2 \underset{(j,k) \in \overrightarrow{E}}{\max} \, |\beta^R_{jk}|$.  
This implies that $g_{jk}^R=\frac{\beta_{jk}^R}{M}$ is the optimal solution. Thus, the objective function reduces to $\ell_1$ regularization with the coefficient $\frac{\lambda_n}{M}$. 

Finally, Proposition \ref{Prop2} establishes that for $M \geq 2 \underset{(j,k) \in \overrightarrow{E}}{\max} \, |\beta^\star_{jk}|$, the objective function value of the continuous relaxations of {MIQP+CP}, {MIQP+LN} and {MIQP+TO} are equivalent. This implies that the continuous relaxations of all formulations are equivalent, which completes the proof. 
\end{proof}

Despite the promising performance of {MIQP+LN}, its continuous relaxation objective function value provides a weak lower bound due to the big-$M$ constraints. To circumvent this issue, a natural strategy is to improve the big-$M$ value. Nonetheless, existing methods which ensure a valid big-$M$ value or heuristic techniques \citep{park2017bayesian,gomez2018mixed} do not lead to tight big-$M$ values. For instance,  the heuristic technique by \cite{park2017bayesian} to obtain big-$M$ values always satisfies the condition in Proposition \revised{\ref{prop:Prop3}} and exact techniques are expected to produce even larger big-$M$ values. Therefore, we  directly develop tighter approximations for \eqref{CP-con1} next.

\section{New Perspectives for  Mathematical Formulations of DAG Learning} \label{Sec: Math models}

In this section, we discuss improved mathematical formulations for learning DAG structure of a BN based on convex (instead of linear) encodings of the constraints in \eqref{CP-con1}. 

Problem \eqref{MIQP1} is an MIQP with non-convex complementarity constraints \eqref{CP-con1}, a class of problems which has received a fair amount of attention from the operations research community over the last decade \citep{frangioni2006perspective, frangioni2007sdp, frangioni2009computational, frangioni2011projected, gomez2018mixed, Liu2021, Wei2021, Wei2022}. There has also been recent interest in leveraging these developments to solve sparse regression problems with $\ell_0$ regularization \citep{pilanci2015sparse, dong2015regularization, xie2018ccp, atamturk2019rank,wei2019convexification}. 

Next, we review applications of MIQPs with complementarity constraints of the form \eqref{CP-con1} for solving sparse regression with $\ell_0$ regularization. \cite{frangioni2011projected} develop a so-called projected perspective relaxation method, to solve the perspective relaxation of mixed-integer nonlinear programming problems with a convex objective function and complementarity constraints. This reformulation requires that the corresponding binary variables are not involved in other constraints. Therefore, it is suitable for $\ell_0$ sparse regression, but cannot be applied for DAG structure learning. \cite{pilanci2015sparse} show how a broad class of $\ell_0$-regularized problems, including sparse regression as a special case, can be formulated exactly as optimization problems. The authors use the Tikhonov regularization term $\mu\|\beta\|_2^2$ and convex analysis to construct an improved convex relaxation using the reverse Huber penalty. In a similar vein, \cite{bertsimas2017sparse} exploit the Tikhonov regularization and develop an efficient algorithm by reformulating the sparse regression mathematical formulation as a saddle-point optimization problem with an outer linear integer optimization problem and an inner dual quadratic optimization problem which is capable of solving high-dimensional sparse regressions. \cite{xie2018ccp} apply the perspective formulation of sparse regression optimization problem with both $\ell_0$  and the Tikhonov regularizations. 
The authors establish that the continuous relaxation of the perspective formulation is equivalent to the continuous relaxation of the formulation given by \cite{bertsimas2017sparse}.
\cite{dong2015regularization} propose perspective relaxation for $\ell_0$ sparse regression optimization formulation and establish that the popular sparsity-inducing concave penalty function known as the minimax concave penalty \citep{zhang2010nearly} and the reverse Huber penalty \citep{pilanci2015sparse} can be obtained as special cases of the perspective relaxation -- thus the relaxations of formulations by \cite{zhang2010nearly,pilanci2015sparse, bertsimas2017sparse, xie2018ccp} are equivalent. The authors  obtain an optimal perspective relaxation that is no weaker than any perspective relaxation. Among the related approaches, the optimal perspective relaxation by \cite{dong2015regularization} is the only one that does not explicitly require the use of  Tikhonov regularization. 

The perspective formulation, which in essence is a fractional non-linear program, can be cast either as a mixed-integer second-order cone program (MISOCP) or a  semi-infinite mixed-integer  linear program (SIMILP). 
Both formulations can  be solved directly by state-of-the-art optimization packages.  
\cite{dong2015regularization} and \cite{atamturk2019rank} solve the continuous relaxations and then use a heuristic approach (e.g., rounding techniques) to obtain a feasible solution (an upper bound). 
In this paper, we directly solve the MISOCP and SIMILP formulations for learning sparse DAG structures.

Next, we present how perspective formulation can be suitably applied for DAG structure learning with $\ell_0$ regularization. We further cast the problem as MISOCP and SIMILP. 
 To this end, we express the objective function \eqref{L-obj} in the following way: 
	\begin{alignat}{3}
	\notag \quad  & \text{tr}[(I- B)(I-B)^{\top}\mathcal{X}^{\top}\mathcal{X}] + \lambda_n \sum_{(j,k) \in \overrightarrow{E}}  g_{jk}\\
		\label{PR-obj4} \quad  & = \text{tr}[(I- B-B^\top)\mathcal{X}^{\top}\mathcal{X} + 2BB^\top \mathcal{X}^{\top}\mathcal{X}\revised{]} + \lambda_n \sum_{(j,k) \in \overrightarrow{E}}  g_{jk}.
	\end{alignat}

\noindent Let $\delta \in \mathbb{R}_{+}^{m}$ be a vector such that $\mathcal{X}^\top\mathcal{X} -{D}_\delta \succeq 0$, where $D_\delta=\text{diag}(\delta_1, \dots, \delta_m)$ and $A \succeq 0$ means that matrix $A$ is positive semi-definite.  By splitting the quadratic term $\mathcal{X}^{\top}\mathcal{X}= (\mathcal{X}^{\top}\mathcal{X} -D_\delta)+D_\delta$ in \eqref{PR-obj4}, the objective function can be expressed as 
\begin{equation}\label{eq:LNform1}
 \text{tr}\left[(I- B-B^\top)\mathcal{X}^{\top}\mathcal{X} + BB^\top(\mathcal{X}^{\top}\mathcal{X}  - D_\delta)\right] + \text{tr}\left(B B^\top D_\delta\right) +\lambda_n \sum_{(j,k) \in \overrightarrow{E}}  g_{jk}.
\end{equation}
Let $Q= \mathcal{X}^{\top}\mathcal{X} - D_{\delta}$. (In the presence of Tikhonov regularization with tuning parameter $\mu> 0$, we let $Q= \mathcal{X}^{\top}\mathcal{X} + \mu I- D_{\delta}$ as described in Remark~\ref{rem:L2}.) Then, Cholesky decomposition can be applied to decompose $Q$ as $q^{\top}q$ (note $Q \succeq 0$). As a result, $\text{tr}\left(B B^\top Q\right) = \text{tr}\left(B B^\top {q^\top} {q}\right) = \sum_{i=1}^{m} \sum_{j=1}^{m} \left(\sum_{(\ell,j) \in \overrightarrow{E}} \beta_{\ell j}q_{i\ell}\right)^2$. The separable component can also be expressed as $\text{tr}\left(B B^\top D_\delta\right) = \sum_{j=1}^{m} \sum_{(j,k) \in \overrightarrow{E}} \delta_j\beta_{jk}^2$. Using this notation, the objective \eqref{eq:LNform1} can be written as  
\begin{equation}
	 \nonumber  \label{Obj} \text{tr}\left[(I- B-B^\top)\mathcal{X}^{\top}\mathcal{X} + BB^\top Q\right] +\sum_{j=1}^{m} \sum_{(j,k) \in \overrightarrow{E}} \delta_j\beta_{jk}^2 + \lambda_n\sum_{(j,k) \in \overrightarrow{E}}  g_{jk}.
\end{equation} 

\noindent The Perspective Reformulation (PRef) of {MIQP} is then given by
\begin{subequations}\label{eq:PR}
	\begin{alignat}{3}
	\textbf{PRef}  \label{PR-Obj} \quad \underset{\revised{B \in {\mathbb R}^{m\times m}, g\in \{0,1\}^{|\overrightarrow{E}|}}}{\min} \quad &\text{tr}\big[(I- B-B^\top)\mathcal{X}^{\top}\mathcal{X} + BB^\top Q\big] + \\ & \nonumber \sum_{j=1}^{m} \sum_{(j,k) \in \overrightarrow{E}} \delta_j \frac{\beta_{jk}^2}{g_{jk}} + \lambda_n\sum_{(j,k) \in \overrightarrow{E}}  g_{jk},\\
	&  	 \eqref{LN-con1}-\eqref{LN-con5}.
	\end{alignat}
\end{subequations} 
The objective function \eqref{PR-Obj} is formally undefined when some $g_{jk}$ = 0. More precisely, we use the convention that $\frac{\beta^2_{jk}}{g_{jk}}=0$ when $\beta_{jk} = g_{jk} = 0$ and $\frac{\beta^2_{jk}}{g_{jk}}=+\infty$ when $\beta_{jk} \neq 0$ and $g_{jk}=0$ \citep{frangioni2009computational}. The continuous relaxation of PRef, referred to as the perspective relaxation, is much stronger than the continuous relaxation of MIQP \revised{under certain conditions discussed in detail in Section \ref{deltavalue}}  \citep{pilanci2015sparse}. However, an issue with PRef is that the objective function is nonlinear due to the fractional term. 
There are two ways to reformulate PRef. One as a mixed-integer second-order conic program (MISOCP) (see, Section \ref{SOCP}) and the other as a   semi-infinite mixed-integer linear program (SIMILP) (see, Section \ref{SIP}).

\subsection{Mixed-integer second-order conic program} \label{SOCP}
Let $s_{jk}$ be additional variables representing $\beta_{jk}^2$. Then, the MISOCP formulation is given by  
\begin{subequations} \label{eq:misocp}
	\begin{alignat}{3}
	\textbf{MISOCP}\quad \underset{\revised{B \in {\mathbb R}^{m\times m}, s\in {\mathbb R}^{|\overrightarrow{E}|}, g\in \{0,1\}^{|\overrightarrow{E}|}}}{\min} \quad &\text{tr}\left[(I- B-B^\top)\mathcal{X}^{\top}\mathcal{X} + BB^\top Q\right] + \\ & \nonumber \sum_{j=1}^{m} \sum_{(j,k) \in \overrightarrow{E}} \delta_j s_{jk} + \lambda_n \sum_{(j,k) \in \overrightarrow{E}}  g_{jk},\\
	& \label{SOCP-C1} s_{jk}g_{jk} \geq \beta_{jk}^2 \quad (j,k) \in \overrightarrow{E},\\
	& \label{SOCP-C2} 0\le s_{jk} \leq M^2 g_{jk} \quad (j,k) \in \overrightarrow{E},\\
		&   \eqref{LN-con1}-\eqref{LN-con5}.
	\end{alignat}
\end{subequations} 
Here, the constraints in \eqref{SOCP-C1} imply that $\beta_{jk} \neq 0$ only when $g_{jk} = 1$. The constraints in \eqref{SOCP-C1} are second-order conic representable because they can be written in the form of $\sqrt{4\beta_{jk}^2+ (s_{jk}-g_{jk})^2}\leq s_{jk}+g_{jk}$. The set of constraints in \eqref{SOCP-C2} is valid since $\beta_{jk} \leq Mg_{jk}$ implies $\beta_{jk}^2 \leq M^2g^2_{jk}= M^2g^2_{jk}$ and $g_{jk}^2=g_{jk}$ for $g_{jk} \in \{0,1\}$. The set of constraints in \eqref{SOCP-C2} is not required, yet they improve the computational efficiency especially when we restrict the big-$M$ value. \cite{xie2018ccp} report similar behavior for sparse regression. 
 When we relax $g_{jk}\in \{0,1\}$ and let  $g_{jk}\in[0,1]$, we obtain the continuous relaxation of {MISOCP} \eqref{eq:misocp}. Let us denote the feasible region of continuous relaxation of {MISOCP} \eqref{eq:misocp} and {MIQP} \eqref{eq:LNform} by $\mathcal{R}$MISOCP and $\mathcal{R}$MIQP, and the objective function values by OFV($\mathcal{R}$MISOCP) and OFV($\mathcal{R}$MIQP), respectively. For a more general problem than ours, \cite{cui2013convex} give a detailed proof establishing that the feasible region of the former is contained in the feasible region of latter i.e., $\mathcal{R}$MISOCP $\subset \mathcal{R}MIQP$, and  OFV($\mathcal{R}$MISOCP) $ \ge $ OFV($\mathcal{R}$MIQP). Therefore, we are able to obtain stronger lower bounds using MISOCP than MIQP \revised{under suitable choices for $D_\delta$ as described  in Section \ref{deltavalue}}. 
 
\subsection{Mixed-integer semi-infinite integer linear program} \label{SIP}
An alternative approach to reformulate PRef is via \textit{perspective cuts} developed by \cite{frangioni2006perspective,frangioni2007sdp}. To apply perspective cuts, we use the reformulation idea first proposed in \cite{frangioni2006perspective} by introducing dummy decision matrix $D$ to distinguish the separable and non-separable part of the objective function; we also add the additional constraint $d = \beta$ where $d_{jk}$ is $(j,k)$ element of matrix $D$ and $\beta$ is the decision variable in the optimization problem. Following this approach,  {MIQP} can be reformulated as an SIMILP: 
\begin{subequations}
 	\begin{alignat}{3}
 	\textbf{SIMILP}\quad  \underset{\revised{B \in {\mathbb R}^{m\times m}, v\in {\mathbb R}^{|\overrightarrow{E}|}, g\in \{0,1\}^{|\overrightarrow{E}|}}}{\min} \quad &\text{tr}\left[(I- B-B^\top)\mathcal{X}^{\top}\mathcal{X} + DD^\top Q\right] + \\ & \nonumber \sum_{j=1}^{m} \sum_{(j,k) \in \overrightarrow{E}} \delta_j v_{jk} + \lambda_n \sum_{(j,k) \in \overrightarrow{E}}  g_{jk}, \\
 	& \label{SIP-C0} d_{jk} = {\beta}_{jk} \quad  (j,k) \in \overrightarrow{E}, \\
 	& \label{SIP-C1} v_{jk} \geq 2 \bar{\beta}_{jk}\beta_{jk} - \bar{\beta}_{jk}^2g_{jk} \quad \forall \bar{\beta}_{jk} \in [-M, M] \quad \forall (j,k) \in \overrightarrow{E}, \\
	& \label{SIP-C2} \eqref{LN-con1}-\eqref{LN-con5}, \\
 	& v_{jk} \geq 0, \quad  (j,k) \in \overrightarrow{E}.
 	\end{alignat}
 \end{subequations} 
The set of constraints in \eqref{SIP-C1} is known as perspective cuts. Note that there are infinitely many such constraints. Although this problem cannot be solved directly, it lends itself to a delayed cut generation approach whereby a (small) finite subset of constraints in \eqref{SIP-C1}  is kept, the current solution $(\beta^{\star}, g^{\star}, v^{\star})$ of the relaxation is obtained, and all the violated inequalities for the relaxation solution are added for $\bar{\beta}_{jk} = \frac{\beta^{\star}_{jk}}{g^{\star}_{jk}}$ (assuming $\frac{0}{0} = 0$). This process is repeated until termination criteria are met. This procedure can be implemented using the cut callback function available by off-the-shelf solvers such as Gurobi or CPLEX. 
   
\subsection{Selecting $\delta$} \label{deltavalue}
In the MISOCP and SIMILP formulations, one important question is how to identify a valid $\delta$. A natural choice is diag$(\delta) = \lambda_{\min} e$, where $\lambda_{\min}$ is the minimum eigenvalue of $\mathcal{X}^\top\mathcal{X}$ and $e$ is a column vector of ones. The issue with this approach is that if $\lambda_{\min} =0$, then $\text{diag}({\delta})$ becomes a trivial 0 matrix. If $\text{diag}(\delta)$ turns out to be a zero matrix, then MISOCP formulation reduces to the big-$M$ formulation. \cite{frangioni2007sdp} present an effective approach for obtaining a valid $\delta$ by solving the following semidefinite program (SDP)  
\begin{subequations}
	\begin{alignat}{3}
\label{delta} \underset{\revised{\delta\in \mathbb{R}^{|V|}}}{\max} \left\{\sum_{i \in V} \delta_i \,:\, \mathcal{X}^\top \mathcal{X}  - \diag(\delta) \succeq 0, \delta_i \geq 0\right\}.
	\end{alignat}
\end{subequations} 
This formulation can attain a non-zero $D_{\delta}$ even if $\lambda_{\min}=0$. Numerical results by \cite{frangioni2007sdp} show that this method compares favorably with the minimum eigenvalue approach. \cite{zheng2014improving} propose an SDP approach, which obtains $D_{\delta}$ such that the continuous relaxation of {MISOCP} \eqref{eq:misocp} is as tight as possible.    

Similar to \cite{dong2015regularization}, our formulation does not require adding a Tikhonov regularization. In this case, PRef is effective when $\mathcal{X}^\top\mathcal{X}$ is sufficiently diagonally dominant. 
When $n \geq m$ and each row of $\mathcal{X}$ is independent, then $\mathcal{X}^\top\mathcal{X}$ is guaranteed to be a positive semi-definite matrix \citep{dong2015regularization}. On the other hand, when $n < m$, $\mathcal{X}^\top\mathcal{X}$ is not full-rank. Therefore, a Tikhonov regularization term should be added with sufficiently large $\mu$  to make $\mathcal{X}^\top\mathcal{X} + \mu I \succeq 0 $ \citep{dong2015regularization} in order to benefit from the strengthening provided by PRef. 

\section{Experiments} \label{Sec: Computational}
In this section, we report the results of our  numerical experiments that compare different formulations and evaluate the effect of different tuning parameters and estimation strategies. Our experiments are performed on a cluster operating on UNIX  with Intel Xeon E5-2640v4 2.4GHz.\ All formulations are implemented in the Python programming language. Gurobi 8.1 is used as the solver. Unless otherwise stated, a time limit of $50m$ (in seconds), where $m$ denotes the number of nodes, and an MIQP relative optimality gap of $0.01$ are imposed across all experiments after which runs are aborted. The \emph{relative} optimality gap is calculated by RGAP$:=\frac{UB(X)-LB(X)}{UB(X)}$ where UB(X) denotes the objective value associated with the best feasible integer solution (incumbent) and LB(X) represents the best obtained lower bound during the branch-and-bound process for the formulation $X \in \{\mathrm{MIQP}, \mathrm{SIMILP}, \mathrm{MISOCP}\}$.

Unless otherwise stated, we assume $\lambda_n=\log(n)$ which corresponds to the Bayesian information criterion  (BIC) score. To select the big-$M$ parameter, $M$, in all formulations we use the proposal of \citet{park2017bayesian}. Specifically, given $\lambda_n$, we solve each problem without cycle prevention constraints and obtain $\beta^R$. We then use the upper bound $M = 2 \underset{(j,k) \in \overrightarrow{E}}{\max} \, |\beta^R_{jk}|$.
Although this value does not guarantee an upper bound for $M$, the results provided in \cite{park2017bayesian} and \cite{manzour2019integer} computationally confirm that this approach gives a large enough value of $M$.  

The goals of our computational study are twofold. First, we compare the various mathematical formulations to determine which gives us the best performance in Subsection~\ref{sec:synth-data}, compare the  sensitivity to the model parameters in Subsection~\ref{lambda}, and the choice of the regularization term in Subsection~\ref{sec:compl2}.  
Second, in Subsection~\ref{sec:compearly} we use the best-performing formulation to investigate the implications of the early stopping condition on the quality of the solution with respect to the true graph. To be able to perform such a study, we use synthetic data so that the true graph is available. \revised{In Subsection \ref{sec:comp-sota}, we compare our algorithm against two state-of-the-art benchmark algorithms on publicly available datasets.}

We use the package \texttt{pcalg} in \texttt{R} to generate random graphs. First, we create a DAG by \texttt{randomDAG} function and assign random arc weights (i.e., $\beta$) from a uniform distribution, $\mathcal{U}[0.1, 1]$.  
Next, the resulting DAG and random coefficients are fed into the \texttt{rmvDAG} function to generate multivariate data based on linear SEMs (columns of matrix $\mathcal X$) with the standard normal error distribution. We consider $m\in\{10,20,30,40\}$ nodes and $n=100$ samples. The average outgoing degree of each node,  denoted by $d$, is set to two. We generate 10 random Erd\H{o}s-R\'enyi graphs for each setting ($m$, $n$, $d$). 
 We observe that in our instances, the minimum eigenvalue of $\mathcal{X}^\top \mathcal{X}$ across all instances is 3.26 and the maximum eigenvalue is 14.21.

Two types of problem instances are considered: (i) a set of instances with known moral graph corresponding to the true DAG; (ii) a set of instances with a complete undirected graph, i.e., assuming no prior knowledge. We refer to the first class of instances as \textit{moral} instances and to the second class as \textit{complete} instances. The observational data, $\mathcal{X}$, for both classes of instances are the same. The function \texttt{moralize(graph)} in the \texttt{pcalg} \texttt{R}-package is used to generated the moral graph from the true DAG. Although the moral graph can be consistently estimated from data using penalized estimation procedures with polynomial complexity
\citep[e.g.,][]{loh2014high}, the quality of moral graph affects all optimization models. Therefore, we use the true moral graph in our experiments, unless otherwise stated.

\subsection{Comparison of Mathematical Formulations} \label{sec:synth-data}

We use the following MIQP-based metrics to measure the quality of a solution: relative optimality gap (RGAP), computation time in seconds (Time), Upper Bound (UB), Lower Bound (LB), objective function value (OFV) of the initial continuous relaxation, and the number of explored nodes in the branch-and-bound tree ($\#$ BB). 
An  in-depth analysis comparing the existing mathematical formulations that rely on linear encodings of the constraints in \eqref{CP-con1} for MIQP formulations is conducted by \cite{manzour2019integer}. The authors conclude that  {MIQP+LN} formulation outperforms the other MIQP formulations, and the promising performance of MIQP+LN can be attributed to its size: (1) {MIQP+LN} has fewer binary variables and constraints than {MIQP+TO}, (2) {MIQP+LN} is a compact (polynomial-sized) formulation in contrast to {MIQP+CP} which has an exponential number of constraints.   Therefore, in this paper, 
we  analyze the formulations based on the convex encodings of  the constraints in \eqref{CP-con1}.

\subsubsection{Comparison of MISOCP formulations} \label{sec:MISOCP-experiments}
We next experiment with MISOCP formulations. For the set of constraints in \eqref{CP-con2}, we use LN, TO, and CP constraints discussed in Section \ref{lit1} resulting in three formulations denoted as {MISOCP+LN}, {MISOCP+TO}, {MISOCP+CP}, respectively. The {MISOCP+TO} formulation fails to find a feasible solution for instances with 30 and 40 nodes, see Table \ref{Details}. For moral instances, the optimality gaps for {MISOCP+TO} are 0.000 and 0.021 for instances with 10 and 20 nodes, respectively; for complete instances, the optimality gap for {MISOCP+TO} formulation are 0.009 and 0.272 for instances with 10 and 20 nodes, respectively.  Moreover, Table \ref{Details} illustrates that {MISOCP+LN} performs better than {MISOCP+TO} for even small instances (i.e., 10 and 20 nodes). 

\begin{table}[t] 
		\caption{Optimality gaps for {MISOCP+TO} and {MISOCP+LN} formulations} \label{Details}
\centering
\footnotesize{
	\begin{tabular}{l|l|ll|l|l}
		\hline
		& \multicolumn{2}{c}{Moral} &  & \multicolumn{2}{c}{Complete} \\ \hline 
		$m$ & {MISOCP+TO}  & {MISOCP+LN} &  & {MISOCP+TO}   & {MISOCP+LN}  \\ \hline 
		10 & 0.000            & 0.000           &  & 0.009        & 0.008       \\
		20 & 0.021        & 0.006      &  & 0.272             & 0.195        \\
		30 & -            & 0.010      &  & -             & 0.195        \\
		40 & -            & 0.042      &  & -             & 0.436   \\ \hline
	\end{tabular} 	
	\\ ``-" denotes that no feasible solution, i.e., UB, is obtained within the time limit, so optimality gap cannot be computed. 
}
\end{table} 

For {MISOCP+CP}, instead of incorporating all constraints given by \eqref{CE}, we begin with no constraint of type \eqref{CE}. Given an integer solution with cycles, we detect a cycle and impose a new cycle prevention constraint to remove the detected cycle. Depth First Search (DFS) can detect a cycle in a directed graph with complexity $O(|V|+|E|)$. Gurobi \texttt{LazyCallback} function  is used, which allows adding cycle prevention constraints in the branch-and-bound algorithm, whenever an integer solution with cycles is found. The same approach is used by \cite{park2017bayesian} to solve the corresponding MIQP+CP. Note that Gurobi solver follows a branch-and-cut implementation and adds many general-purpose and special-purpose cutting planes.

Figures \ref{Figurea: MISOCP} and \ref{Figureb: MISOCP} show that {MISOCP+LN} outperforms {MISOCP+CP} in terms of relative optimality gap and computational time. In addition, {MISOCP+LN} attains better upper and lower bounds than {MISOCP+CP} (see, Figures \ref{Figurec: MISOCP} and \ref{Figured: MISOCP}).  {MISOCP+CP} requires the solution of a second-order cone program (SOCP) after each cut, which reduces its computational efficiency and results in higher optimality gaps than {MISOCP+LN}. {MISOCP+TO} requires many binary variables which makes the problem very inefficient when the network becomes denser and larger as shown in Table \ref{Details}.  Therefore, we do not illustrate the {MISOCP+TO} results in Figure \ref{Figure: MIQP}.

\begin{figure*}[t!]
	\begin{subfigure}[t]{0.49\textwidth}
		\centering
		\includegraphics[scale=0.22]{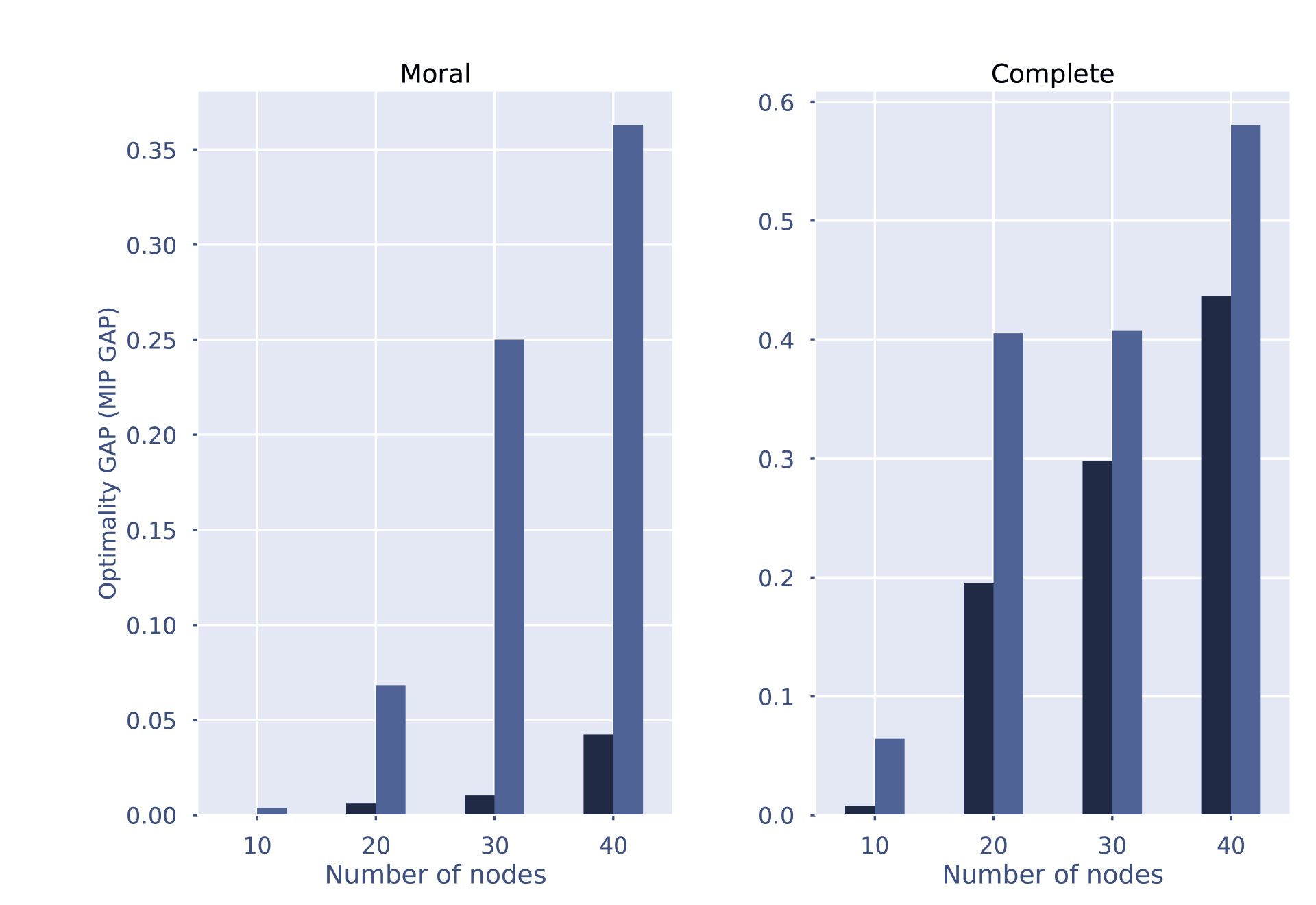}
		\caption{RGAPs}
				\label{Figurea: MISOCP}
	\end{subfigure}%
	~ 
	\begin{subfigure}[t]{0.49\textwidth}
		\centering
		\includegraphics[scale=0.22]{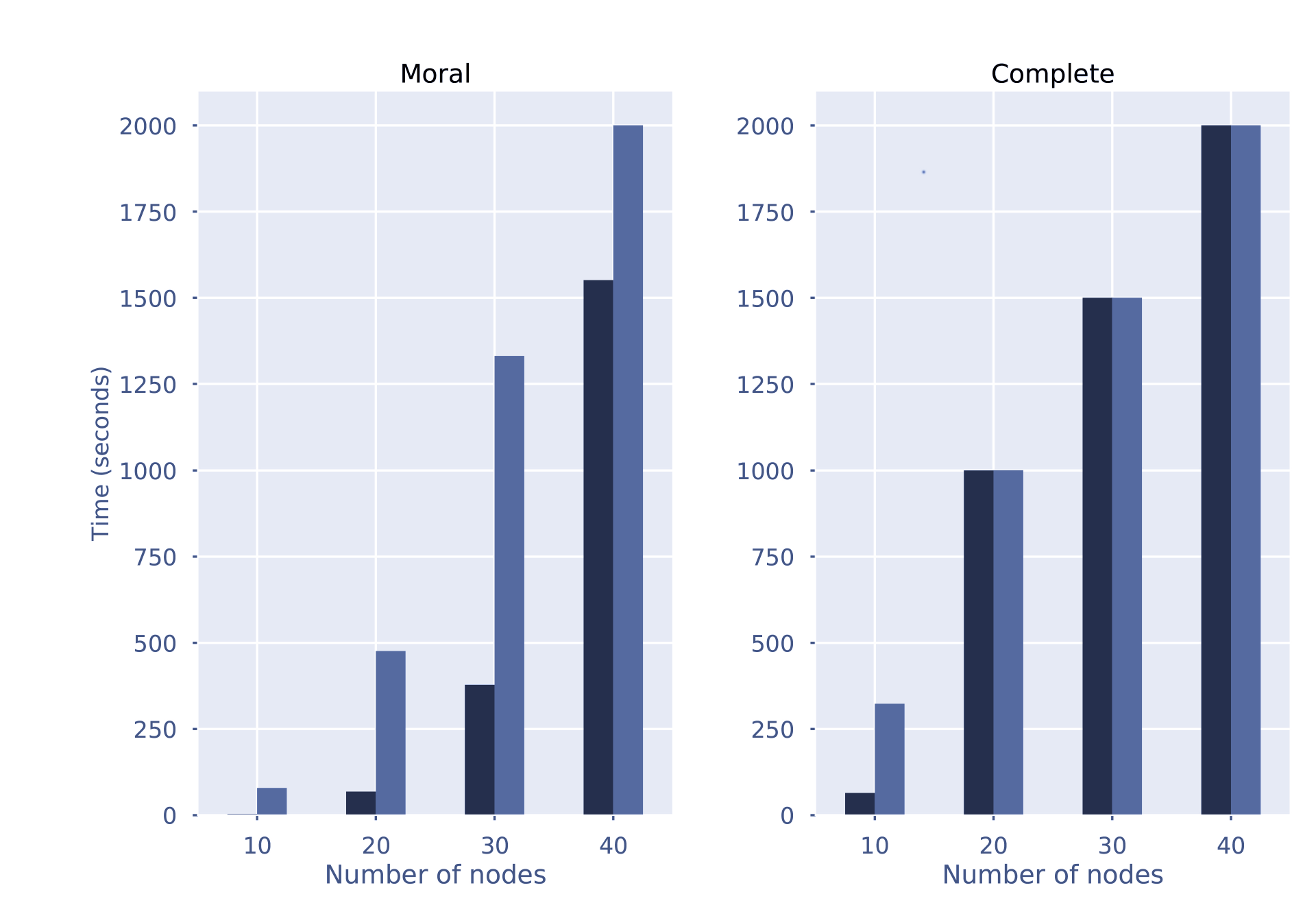}
		\caption{Time (in seconds)}
			\label{Figureb: MISOCP}
	\end{subfigure}
	~
	\begin{subfigure}[t]{0.49\textwidth}
		\centering
		\includegraphics[scale=0.22]{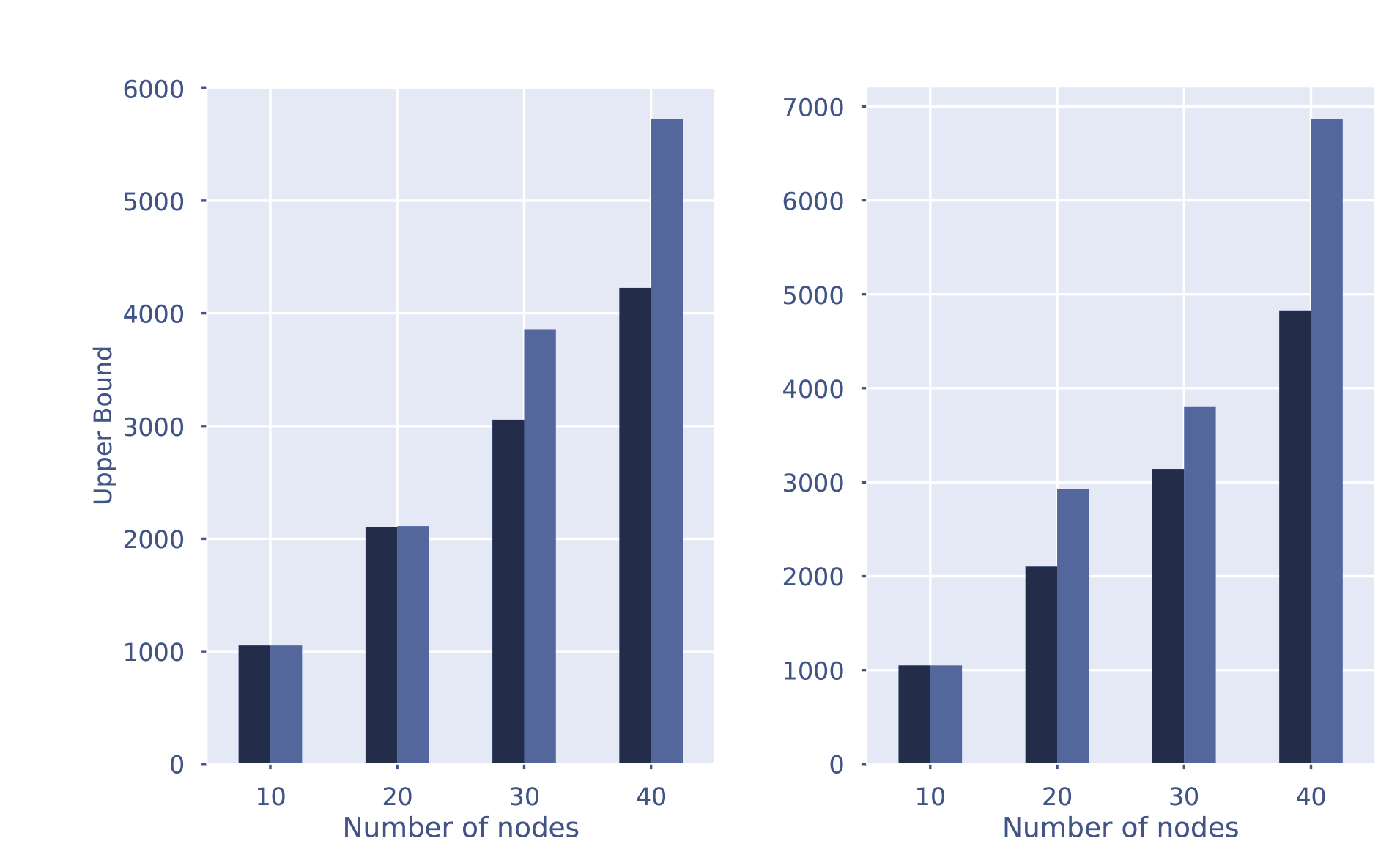}
		\caption{Best upper bounds}
					\label{Figurec: MISOCP}
	\end{subfigure}
	~ 
	\begin{subfigure}[t]{0.49\textwidth}
		\centering
		\includegraphics[scale=0.22]{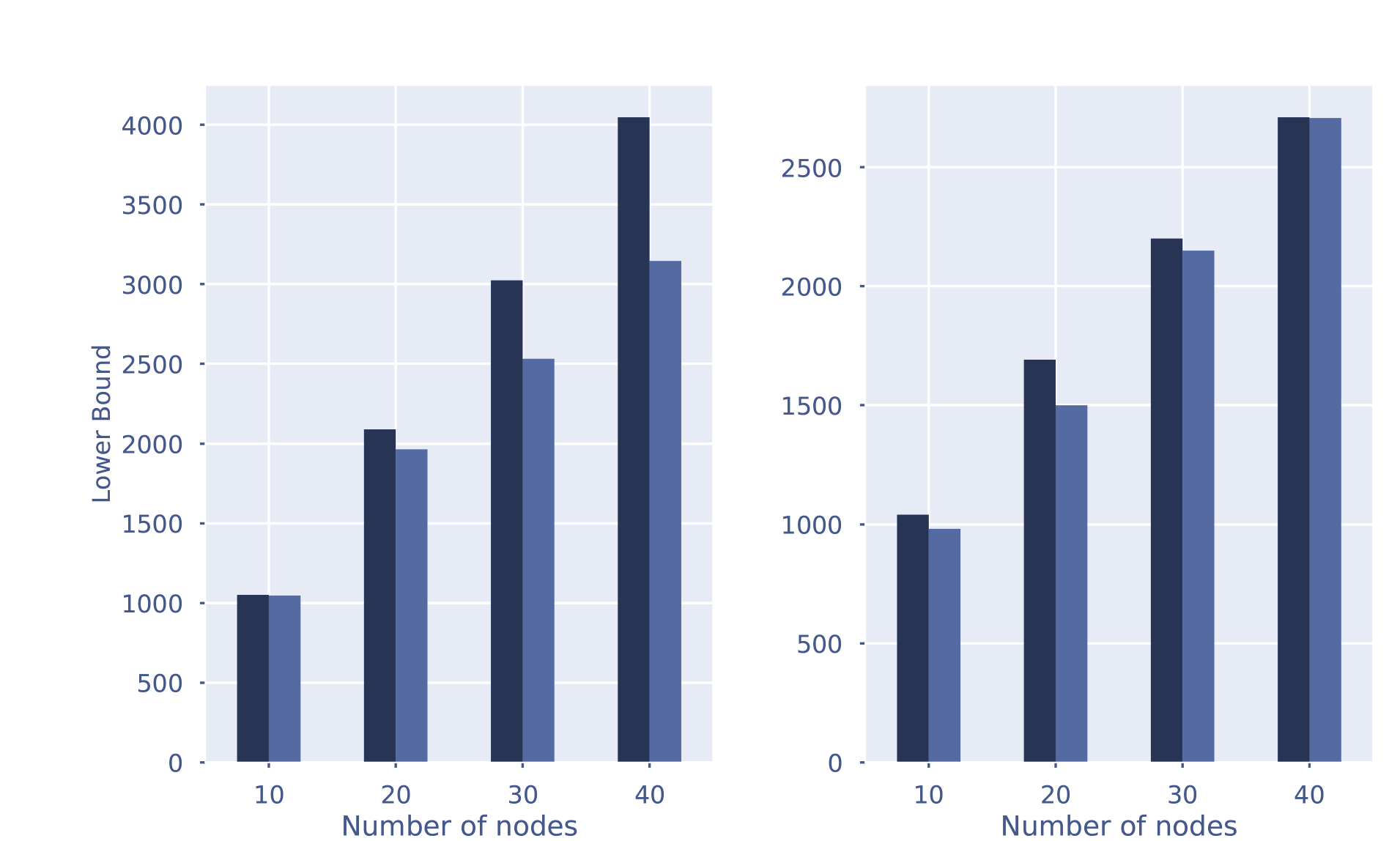}
		\caption{Best lower bounds}
					\label{Figured: MISOCP}
	\end{subfigure}
	~
	\caption{Optimization-based measures for MISOCP+LN (left bar) and MISOCP+CP (right bar) formulations for $n=100$.} 
				\label{Figure: MIQP}
\end{figure*}

\begin{figure*}[t!]
	\begin{subfigure}[t]{0.49\textwidth}
		\centering
		\includegraphics[scale=0.22]{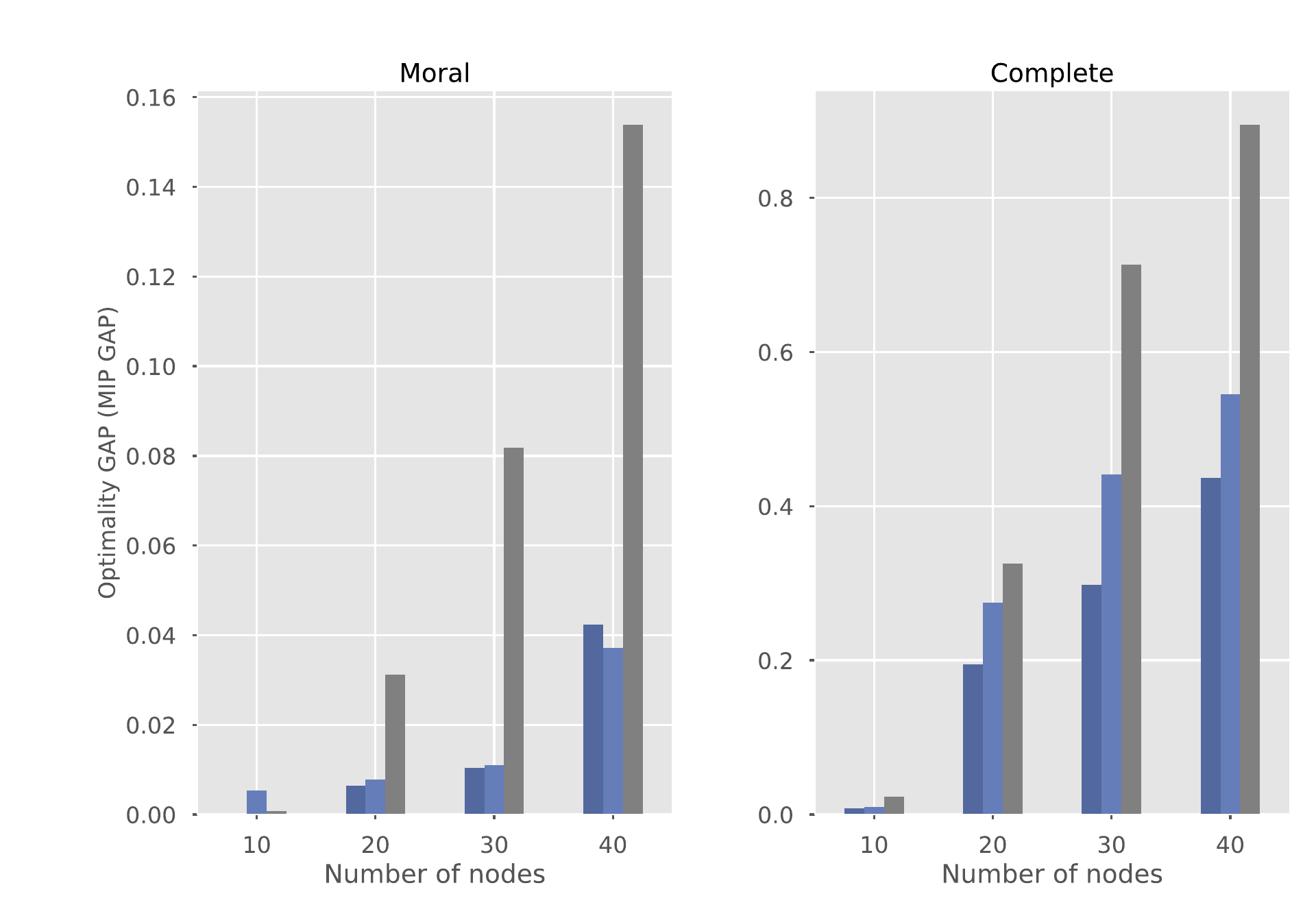}
		\caption{RGAPs}
		\label{Figurea: SIMILP}
	\end{subfigure}%
	~ 
	\begin{subfigure}[t]{0.49\textwidth}
		\centering
		\includegraphics[scale=0.22]{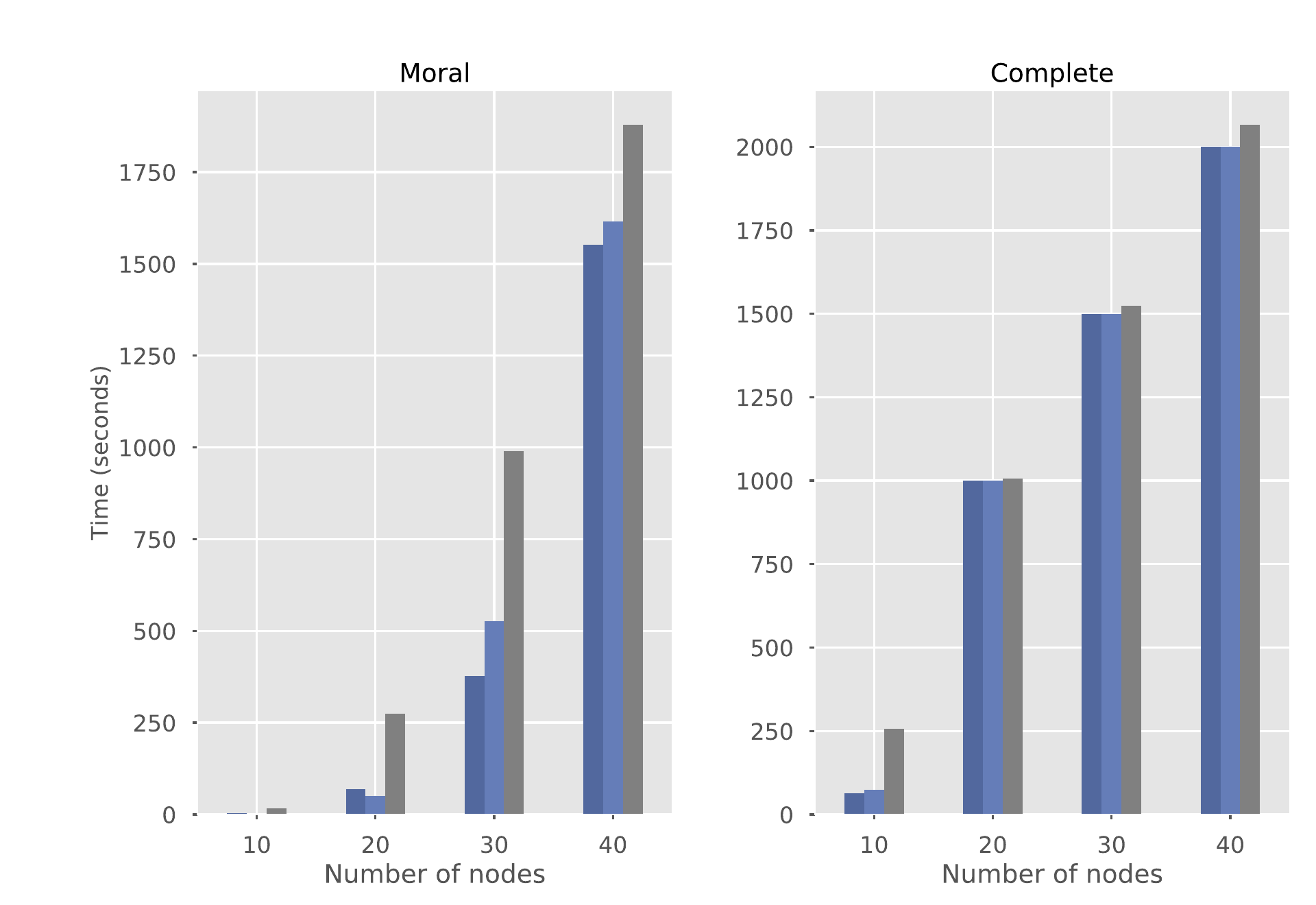}
		\caption{Time (in seconds)}
		\label{Figureb: SIMILP}
	\end{subfigure} 	
		\begin{subfigure}[t]{0.49\textwidth}
			\centering
			\includegraphics[scale=0.22]{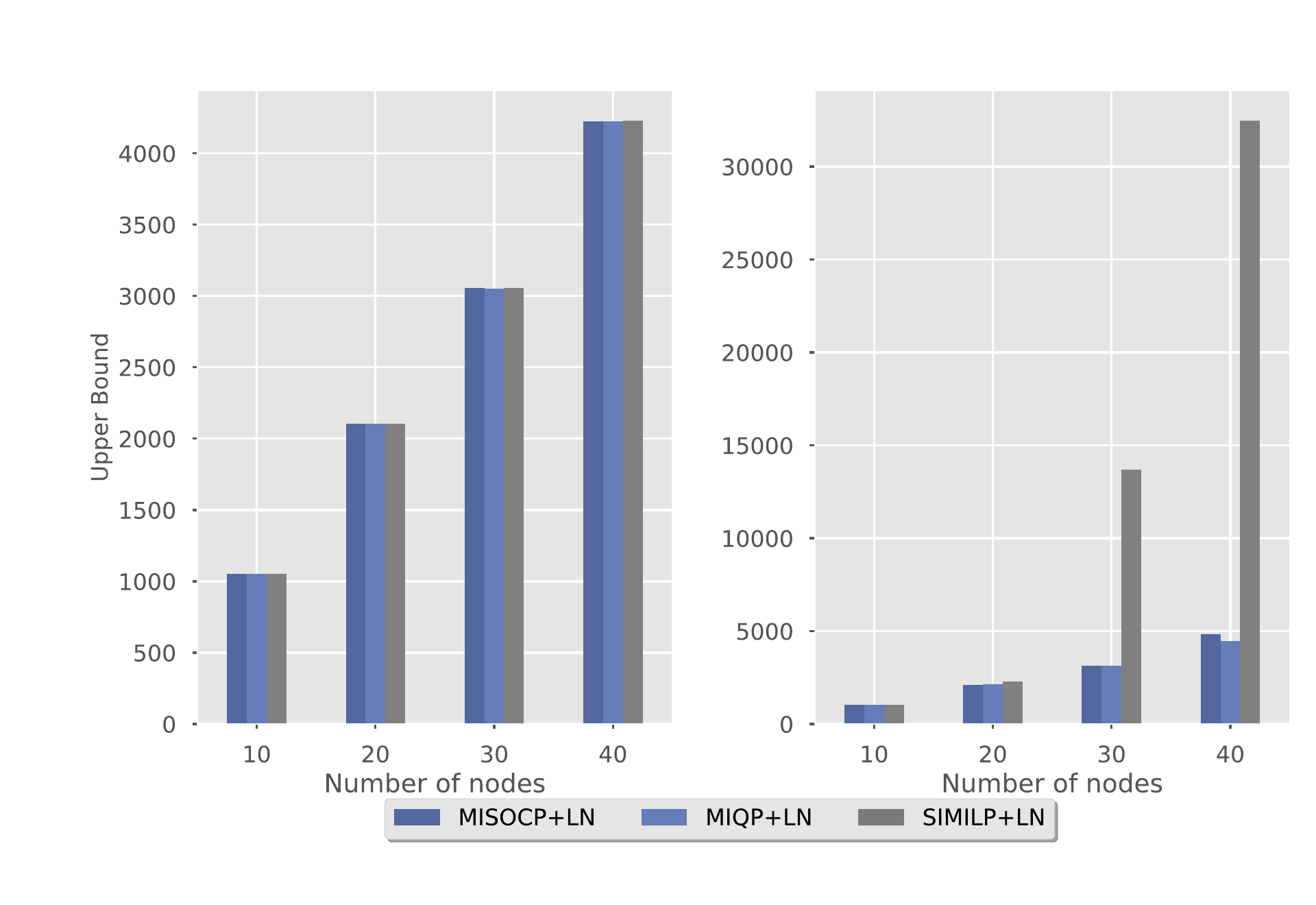}
			\caption{Best upper bounds}
			\label{Figurec: SIMILP}
		\end{subfigure}
		~ 
		\begin{subfigure}[t]{0.49\textwidth}
			\centering
			\includegraphics[scale=0.22]{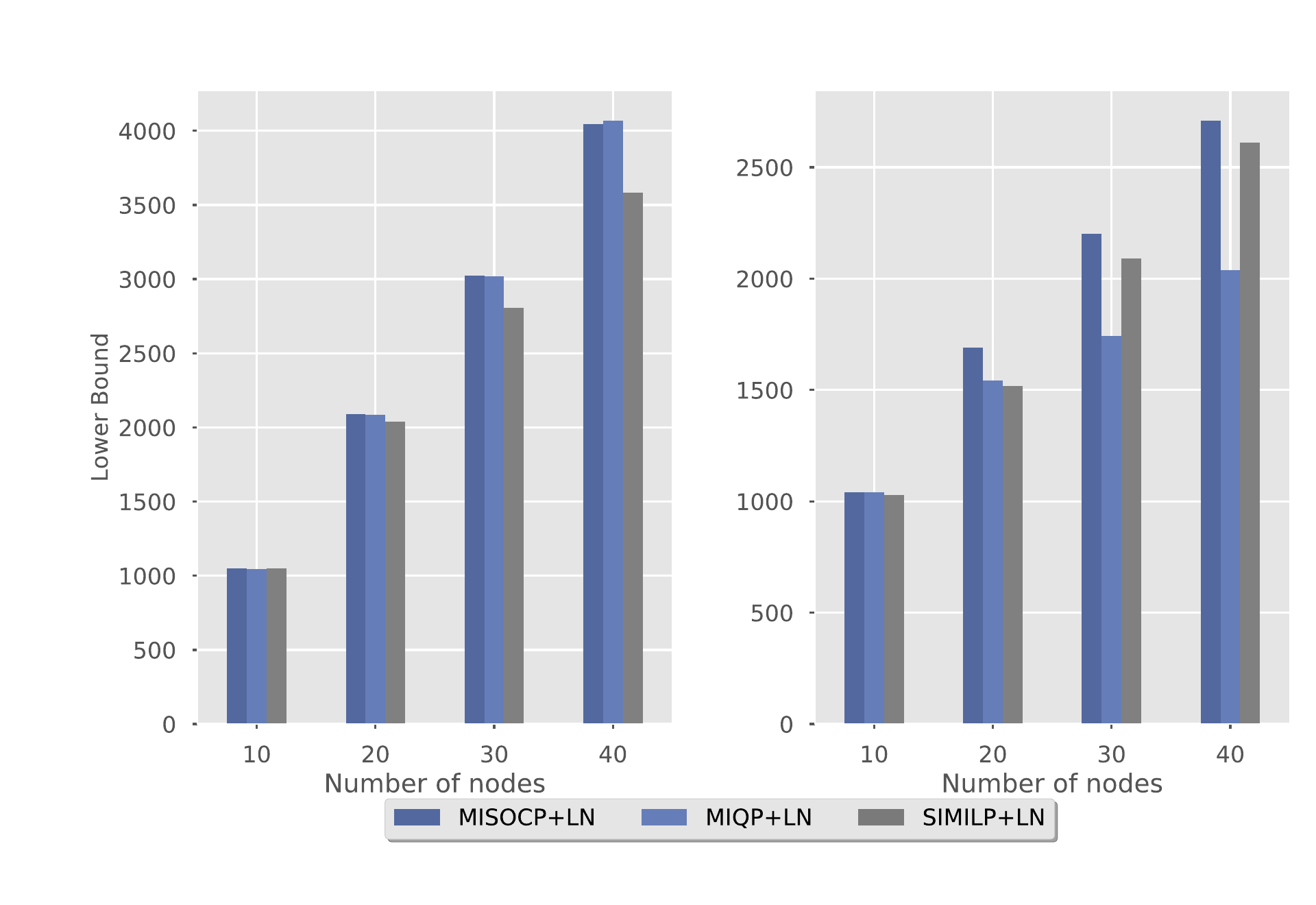}
			\caption{Best lower bounds}
			\label{Figured: SIMILP}
		\end{subfigure}
	\caption{Optimization-based measures for {MISOCP+LN}, {MIQP+LN}, and {SIMILP+LN} formulations for $n=100$.} 
	\label{Figure: SIMLP}
\end{figure*}

\begin{figure*}[t!]
	\begin{subfigure}[t]{0.49\textwidth}
		\centering
		\includegraphics[scale=0.22]{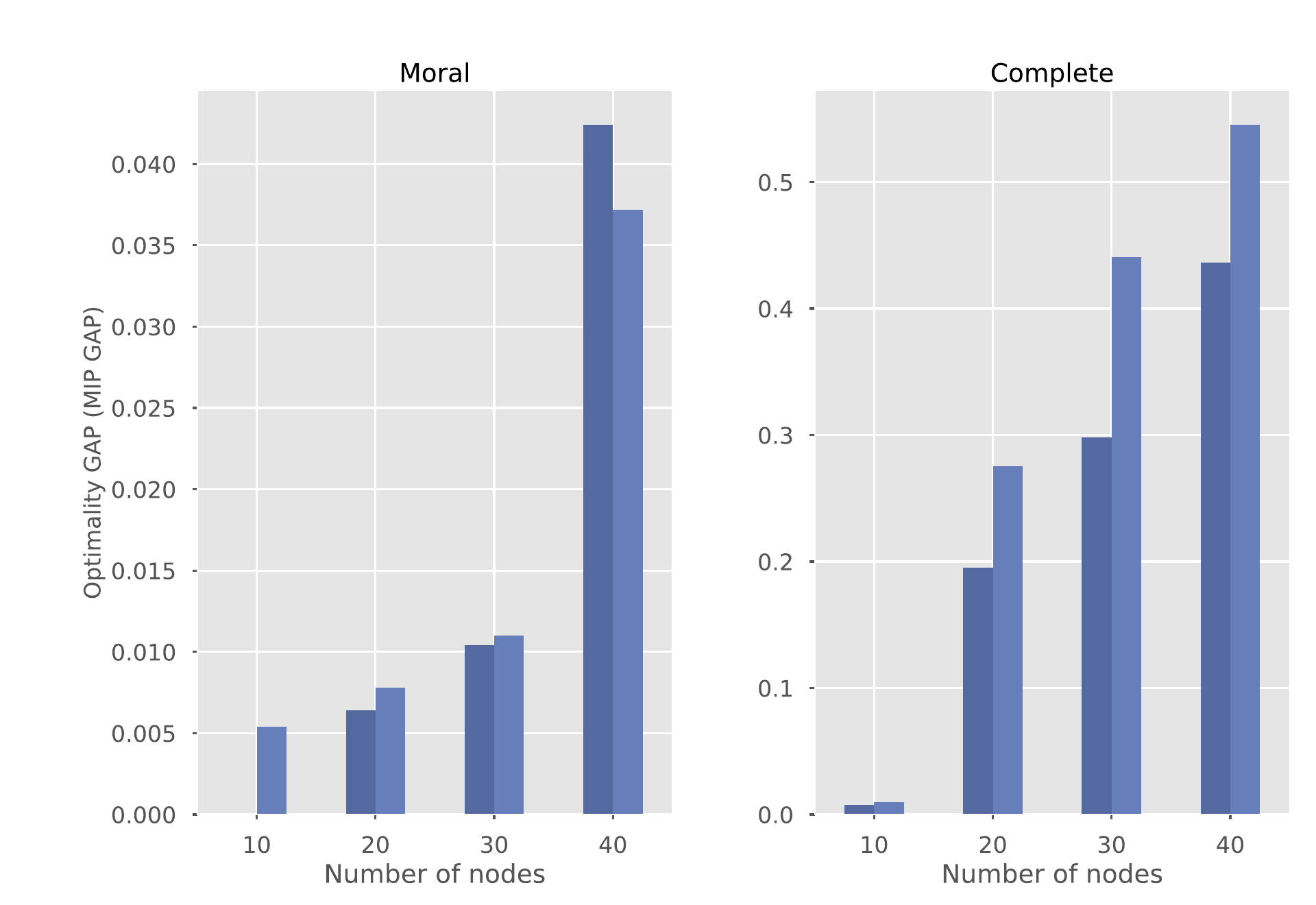}
		\caption{RGAPs}
		\label{Figurea: Best}
	\end{subfigure}%
	~ 
	\begin{subfigure}[t]{0.49\textwidth}
		\centering
		\includegraphics[scale=0.22]{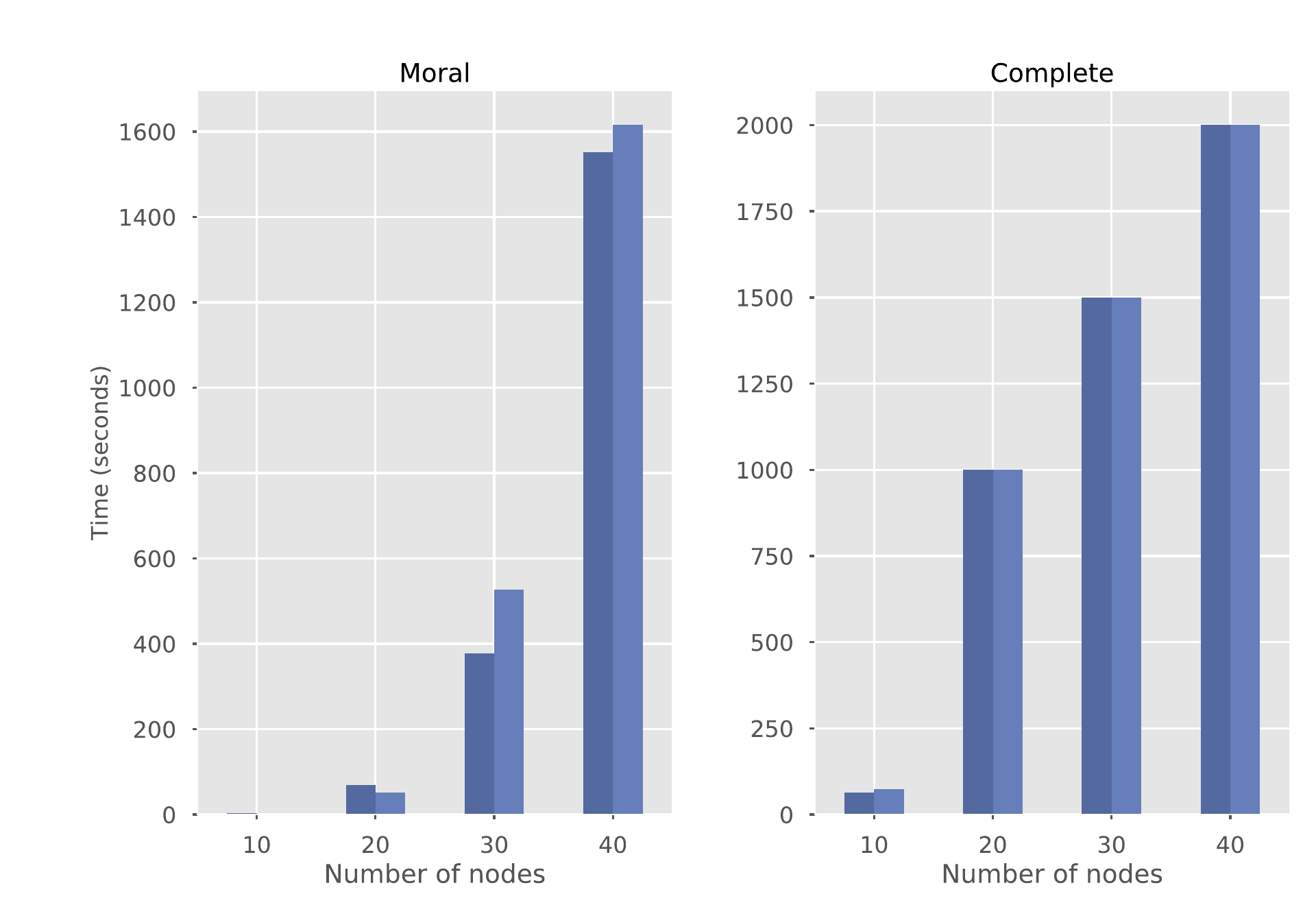}
		\caption{Time (in seconds)}
		\label{Figureb: Best}
	\end{subfigure} 	
	\begin{subfigure}[t]{0.49\textwidth}
		\centering
		\includegraphics[scale=0.22]{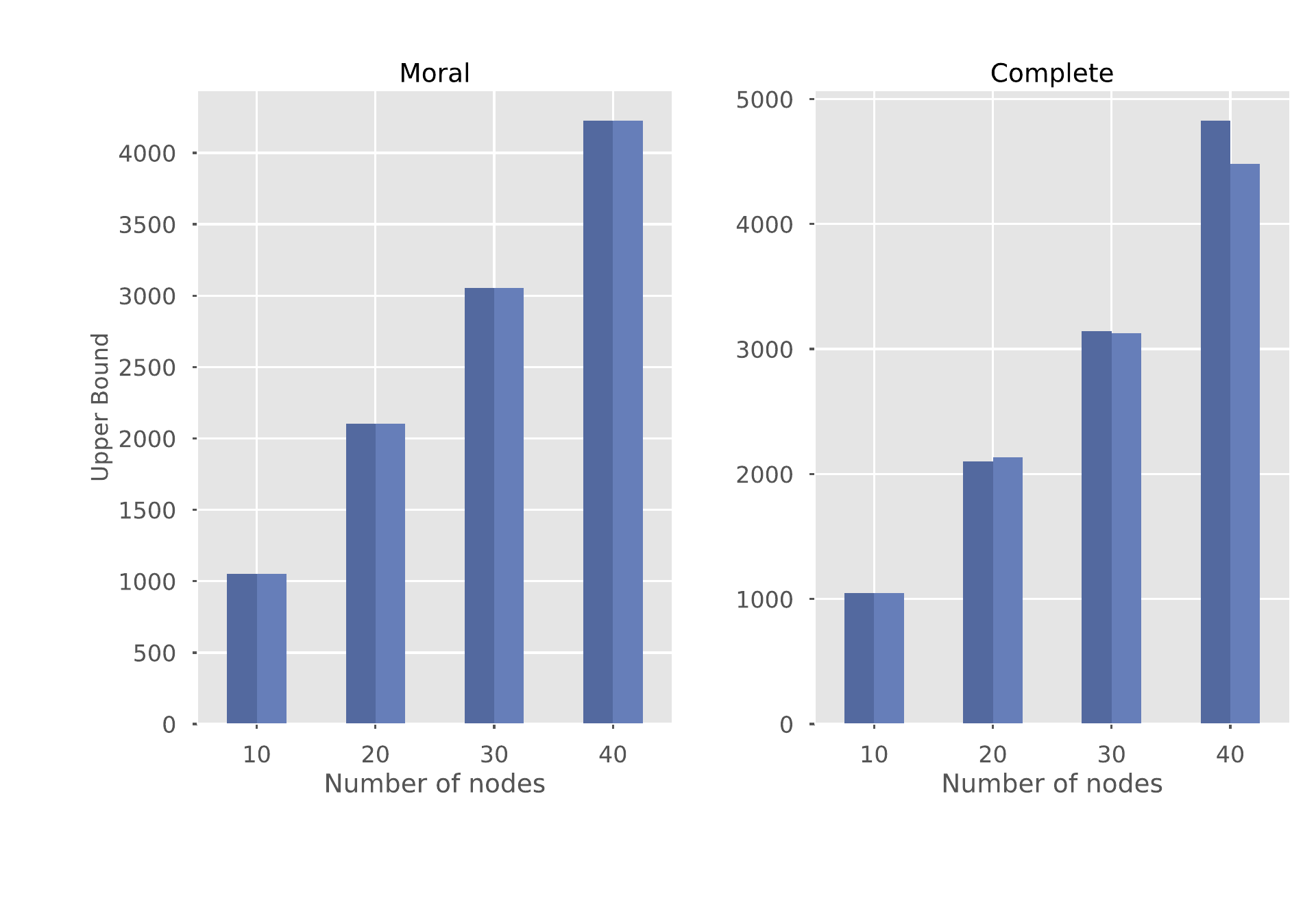}
		\caption{Best upper bounds}
				\label{Figurec: Best}
	\end{subfigure}
	~ 
	\begin{subfigure}[t]{0.49\textwidth}
		\centering
		\includegraphics[scale=0.22]{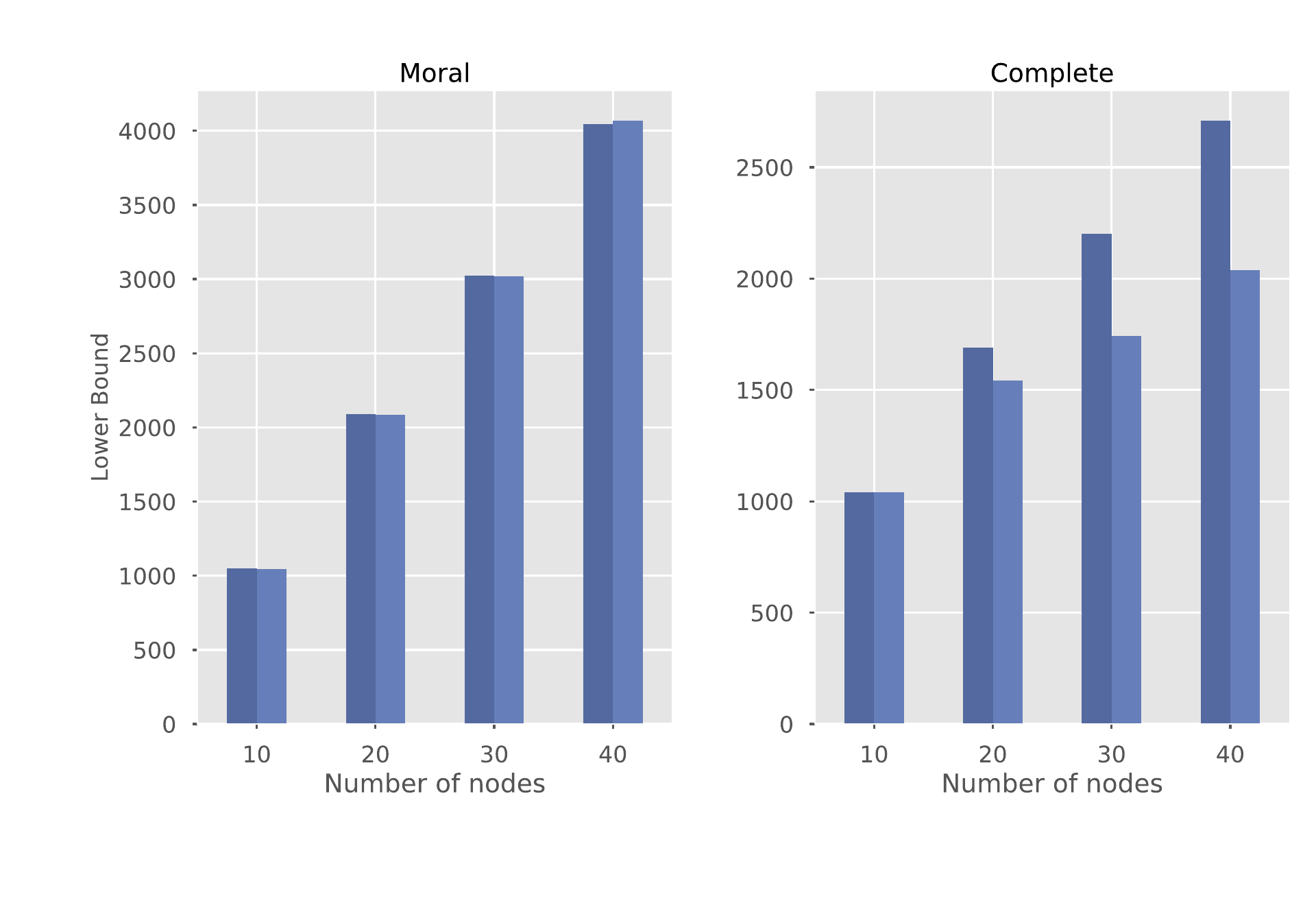}
		\caption{Best lower bounds}
				\label{Figured: Best}
	\end{subfigure}
	\begin{subfigure}[t]{0.49\textwidth}
			\centering
			\includegraphics[scale=0.22]{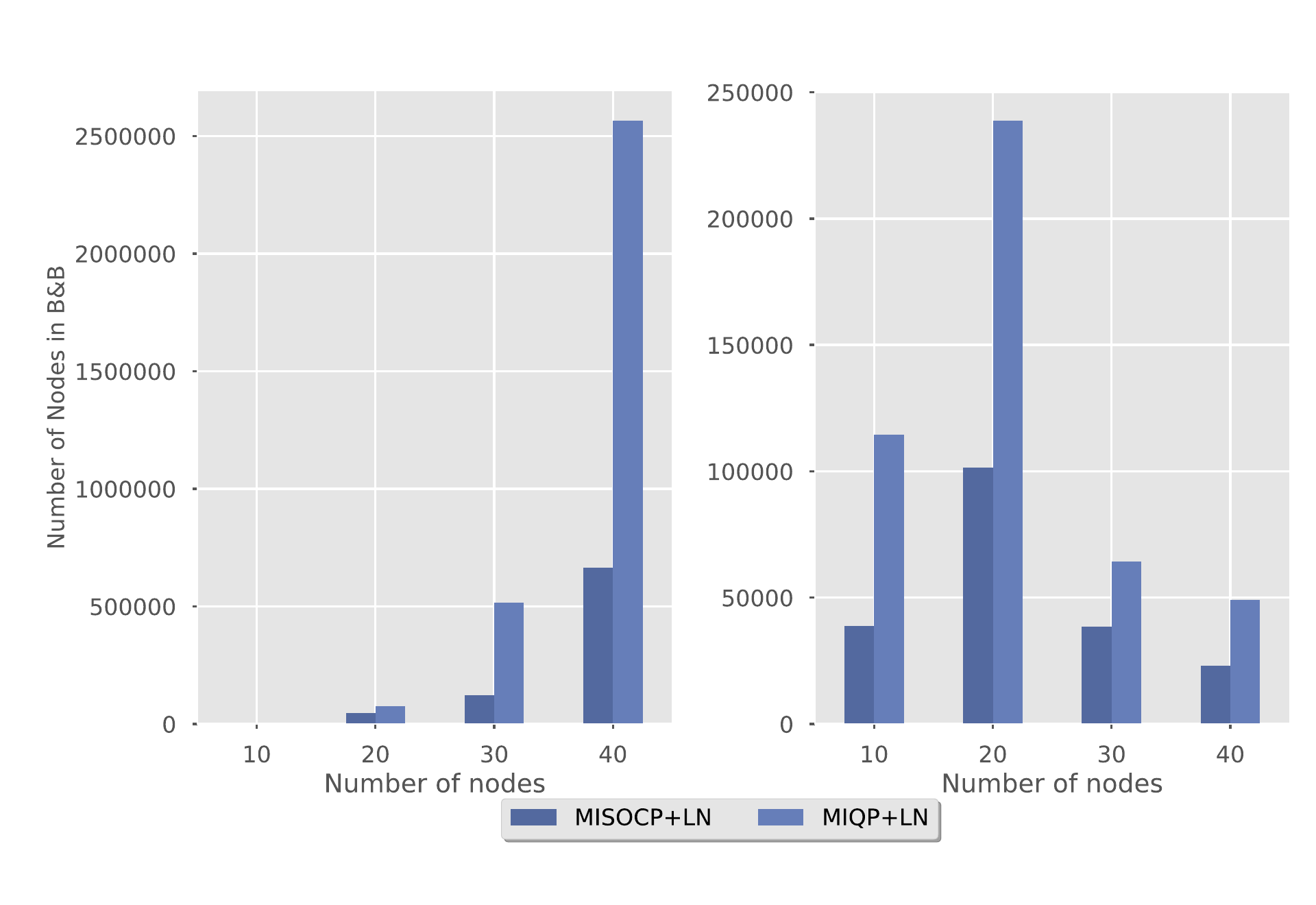}
			\caption{Number of Branch and Bound nodes}
					\label{Figuree: Best}
	\end{subfigure}
		~ 
	\begin{subfigure}[t]{0.49\textwidth}
			\centering
			\includegraphics[scale=0.22]{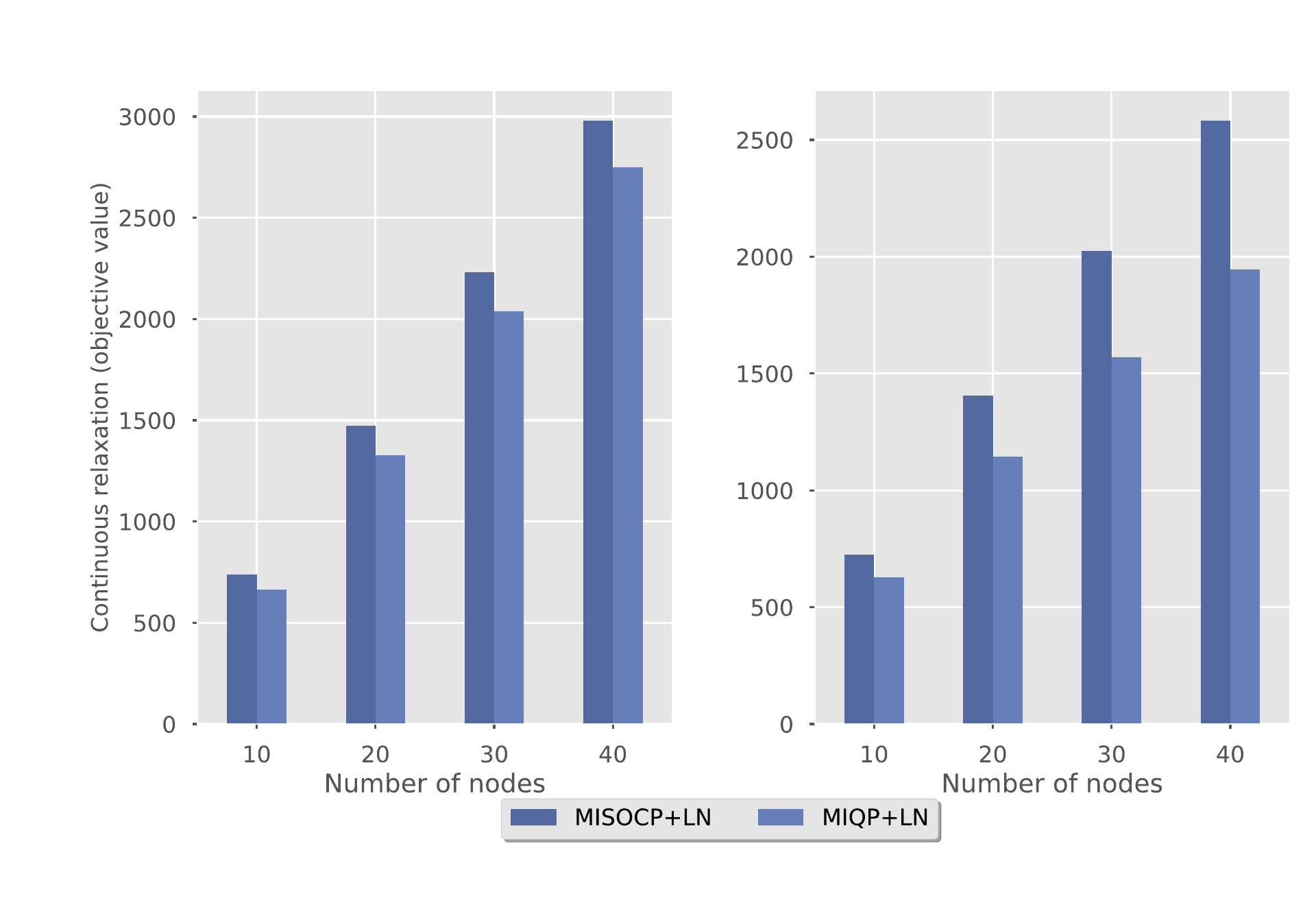}
			\caption{Continuous relaxation objective function}
					\label{Figuref: Best}
	\end{subfigure}
			\label{Figure: Best}
			
			\caption{Optimization-based measures for {MISOCP+LN}, {MIQP+LN} formulations for $n=100$.} 
			\label{Figure: MISOCP-MIQP}
\end{figure*}

\subsubsection{Comparison of MISOCP versus SIMILP} \label{sec:MISOCP ver MIMILP-experiments}
Our computational experiments show that the SIMILP formulation generally performs poorly when compared to {MISOCP+LN} and {MIQP+LN} in terms of optimality gap, upper bound, and computational time. We report the results for {SIMILP+LN}, {MISOCP+LN}, and {MIQP+LN} formulations in Figure \ref{Figure: SIMLP}.  We only consider the LN formulation because that is the best performing model among the alternatives both for MISOCP and MIQP formulations. 

Figures \ref{Figurea: SIMILP} and \ref{Figureb: SIMILP} show the relative optimality gaps and computational times for these three formulations. Figures \ref{Figurec: SIMILP} and \ref{Figured: SIMILP} demonstrate that {SIMILP+LN} attains lower bounds that are comparable with other two formulations. In particular, for complete instances with large number of nodes, {SIMILP+LN} attains better lower bounds than {MIQP+LN}. Nonetheless, {SIMILP+LN} fails to obtain good upper bounds. Therefore, the relative optimality gap is considerably larger for {SIMILP+LN}. 

The poor performance of {SIMILP+LN} might be because state-of-the-art optimization packages (e.g., Gurobi, CPLEX) use many heuristics to obtain a good feasible solution (i.e., upper bound) for a compact formulation. In contrast, SIMILP is not a compact formulation, and we build the SIMILP gradually by adding violated constraints iteratively. Therefore, a feasible solution to the original formulation is not available while solving the relaxations with a subset of the constraints \eqref{SIP-C1}. Moreover, the optimization solvers capable of solving MISOCP formulations have witnessed noticeable improvement due to theoretical developments in this field. In particular, Gurobi reports 20\% and 38\% improvement  in solution time for versions 8 and 8.1, respectively. In addition, Gurobi v8.1 reports over four times faster solution times than CPLEX for solving MISOCP on their benchmark instances.

\subsubsection{Comparison of MISOCP versus MIQP formulations} \label{sec:MISOCP ver MIQP experiments}
In this section, we demonstrate the benefit of using the second-order conic formulation {MISOCP+LN} instead of the linear big-$M$ formulation {MIQP+LN}. As before, we only consider the LN formulation for this purpose.  Figures \ref{Figurea: Best} and \ref{Figureb: Best} show that  {MISOCP+LN} performs better than MIQP+LN in terms of the average relative optimality gap across all number of nodes $m \in \{10,20,30,40\}$. The only exception is $m=40$ for moral instances, for which {MIQP+LN} performs better than {MISOCP+LN}. Nonetheless, we observe that {MISOCP+LN} clearly outperforms {MIQP+LN} for complete instances which are more difficult to solve.  

Figures \ref{Figurec: Best} and \ref{Figured: Best} show the performance of both formulations in terms of the resulting upper and lower bounds on the objective function. We  observe that {MISOCP+LN} attains better lower bounds especially for complete instances. However, {MISOCP+LN}  cannot always obtain a better upper bound. In other words, {MISOCP+LN} is more effective in improving the lower bound instead of the upper bound as expected.

Figures \ref{Figuree: Best} and \ref{Figuref: Best} show that {MISOCP+LN} uses fewer branch-and-bound nodes and achieves   better continuous relaxation values than {MIQP+LN}.  

\subsection{Analyzing the Choices of $\lambda_n$ and $M$} \label{lambda}
We now experiment on different values for $\lambda_n$ and $M$ to assess the effects of these parameters on the performance of {MISOCP+LN} and {MIQP+LN}. First, we consider multiple $\lambda$ values, $\lambda_n \in \{\log{(n)}, 2\log(n), 4\log(n)\}$, while keeping the value of  $M$ the same (i.e.,  $M=2\underset{(j,k) \in \overrightarrow{E}}{\max} \, |\beta^\star_{jk}|$). Table \ref{Table: lambda} shows that as $\lambda_n$ increases, {MISOCP+LN} consistently performs better than  MIQP+LN in terms of the relative optimality gap, computational time, the number of branch-and-bound nodes, and continuous relaxation objective function value. Indeed, the difference becomes even more pronounced for more difficult cases (i.e., complete instances). For instance, for $\lambda_n = 4 \log(n)=18.4$, the relative optimality gap reduces from 0.465 to 0.374, an over 24\% improvement. \revised{In addition, {MISOCP+LN} allows more instances to be solved to optimality within the time limit. For example, for  moral instances with $m = 40$, $\lambda_n = 18.4$, eight out of ten instances are solved to optimality using {MISOCP+LN} while only two instances are solved to optimality by {MIQP+LN}.}

\begin{table}[t!]
	\fontsize{60}{30}\selectfont
	\caption{Computational results for different values of $\lambda_n = t \log(n)$ for $t \in \{1,2,4\}$, * indicates that the problem is solved to the optimality tolerance. Superscript $^{i}$ indicates that out of ten runs, $i$ instances finish before hitting the time limit. Time is averaged over instances that solve within the time limit, RGAP is averaged over instances that reach the time limit.  Better RGAPs are in bold.} \label{Table: IP}
	\resizebox{1\textwidth}{!}{
		\begin{adjustbox}{}{}
			\begin{tabular}{llllllllllllllll|lllllllllllllll}							\\ 	 	\Xhline{2\arrayrulewidth}  \Xhline{2\arrayrulewidth} 
				&&& \multicolumn{11}{c}{Moral} &&& \multicolumn{11}{c}{Complete} \\ 
				\Xhline{2\arrayrulewidth}  \Xhline{2\arrayrulewidth} 
				& \multicolumn{2}{c}{Instances} & & \multicolumn{2}{c}{RGAP}  & & \multicolumn{2}{c}{Time} & & \multicolumn{2}{c}{$\#$ nodes} & & \multicolumn{2}{c}{Relaxation OFV} & & \multicolumn{2}{c}{RGAP}  & & \multicolumn{2}{c}{Time} & & \multicolumn{2}{c}{$\#$ nodes} & & \multicolumn{2}{c}{Relaxation OFV} \\ 
				& $m$ &$\lambda_n$  &  & MISOCP & MIQP &  & MISOCP & MIQP & & MISOCP & MIQP &  & MISOCP & MIQP  & &  MISOCP & MIQP &  & MISOCP & MIQP & & MISOCP & MIQP &  & MISOCP & MIQP  \\
				\Xhline{2\arrayrulewidth} 
				&10 & 4.6 &  & * & * &  & 3 & 2 &  & 1306 & 3715 &  & 738.7 & 664.9 &  & * & * &  & 65 & 74 &  & 38850 & 114433 &  & 724.4 & 629.3 \\ 
&10 & 9.2 &  & * & * &  & 4 & 3 &  & 1116 & 2936 &  & 784.6 & 693.5 &  & * & * &  & 31 & 39 &  & 15736 & 55543 &  & 772.5 & 662.2 \\ 
&10 & 18.4 &  & * & * &  & 3 & 2 &  & 1269 & 2457 &  & 857.0 & 747.5 &  & * & * &  & 26 & 29 &  & 18223 & 41197 &  & 844.5 & 720.2 \\ 
&20 & 4.6 &  & * & * &  & 69 & 51 &  & 46513 & 76261 &  & 1474.2 & 1325.8 &  &\textbf{.195} & .275 &  & 1000 & 1000 &  & 101509 & 238765 &  & 1404.9 & 1144.5 \\ 
&20 & 9.2 &  & * & * &  & 26 & 27 &  & 10695 & 31458 &  & 1589.6 & 1406.8 &  & \textbf{.152} & .250 &  & 1000 & 1000 &  & 152206 & 274514 &  & 1526.9 & 1238.6 \\ 
&20 & 18.4 &  & * & * &  & 24 & 36 &  & 9574 & 33788 &  & 1763.7 & 1552.7 &  &\textbf{.113}$^{\revised{2}}$& .208 &  & $944$ & 1000 &  & 159789 & 277687 &  & 1697.1 & 1395.0 \\ 
&30 & 4.6 &  & \textbf{.010}$^{\revised{8}}$ & 0.011$^{\revised{8}}$ &  & $378$ & $527$ &  & 121358 & 514979 &  & 2230.1 & 2037.7 &  & \textbf{.298} & .441 &  & 1500 & 1500 &  & 38474 & 64240 &  & 2024.0 & 1569.7 \\ 
&30 & 9.2 &  & * & * &  & 104 & 291 &  & 33371 & 248190 &  & 2392.4 & 2168.5 &  & \textbf{.239} & .395 &  & 1500 & 1500 &  & 59034 & 71475 &  & 2217.5 & 1741.5 \\ 
&30 & 18.4 &  & * & * &  & 48 & 74 &  & 15649 & 57909 &  & 2608.3 & 2383.8 &  & \textbf{.215} & .318 &  & 1500 & 1500 &  & 74952 & 96586 &  & 2449.2 & 2006.9 \\ 
&40 & 4.6 &  & .042$^{\revised{6}}$ & \textbf{.037}$^{\revised{4}}$&  & $1551$ & $1615$ &  & 664496 & 2565247 &  & {2979.3} & 2748.6 &  & \textbf{.436} & .545 &  & 2000 & 2000 &  & 23083 & 49050 &  & 2582.0 & 1946.3 \\ 
&40 & 9.2 &  & \textbf{.024}$^{\revised{8}}$ & .036$^{\revised{4}}$ &  & $1125$ & $1336$ &  & 353256 & 1347702 &  & 3200.7 & 2923.5 &  & \textbf{.397} & .473 &  & 2000 & 2000 &  & 29279 & 73917 &  & 2869.9 & 2216.9 \\ 
&40 & 18.4 &  & \textbf{.024}$^{\revised{8}}$ & .035$^{\revised{2}}$ &  & $1099$ & $1375$ &  & 434648 & 1137666 &  & 3521.8 & 3225.4 &  & \textbf{.374} & .465 &  & 2000 & 2000 &  & 31298 & 60697 &  & 3240.1 & 2633.1 \\
				\Xhline{2\arrayrulewidth}
			\end{tabular}
	\end{adjustbox}}
	\vskip 2ex
	\label{Table: lambda}
\end{table}

\begin{table}[t!]
	\fontsize{60}{30}\selectfont
	\caption{Computational results for different values of $\gamma$, * indicates that the problem is solved to the optimality tolerance. Superscript $^{i}$ indicates that out of ten runs, $i$ instances finish before hitting the time limit. Time is averaged over instances that solve within the time limit, RGAP is averaged over instances that reach the time limit. Better RGAPs are in bold.} \label{Table: IP}
	\resizebox{1\textwidth}{!}{
		\begin{adjustbox}{}{}
			\begin{tabular}{llllllllllllllll|lllllllllllllll}							\\ 	 	\Xhline{2\arrayrulewidth}  \Xhline{2\arrayrulewidth} 
				&&& \multicolumn{11}{c}{Moral} &&& \multicolumn{11}{c}{Complete} \\ 
				\Xhline{2\arrayrulewidth}  \Xhline{2\arrayrulewidth} 
				& \multicolumn{2}{c}{Instances} & & \multicolumn{2}{c}{RGAP}  & & \multicolumn{2}{c}{Time} & & \multicolumn{2}{c}{$\#$ nodes} & & \multicolumn{2}{c}{Relaxation OFV} & & \multicolumn{2}{c}{RGAP}  & & \multicolumn{2}{c}{Time} & & \multicolumn{2}{c}{$\#$ nodes} & & \multicolumn{2}{c}{Relaxation OFV} \\ 
				& $m$ &$\gamma$  &  & MISOCP & MIQP &  & MISOCP & MIQP & & MISOCP & MIQP &  & MISOCP & MIQP  & &  MISOCP & MIQP &  & MISOCP & MIQP & & MISOCP & MIQP &  & MISOCP & MIQP  \\
				\Xhline{2\arrayrulewidth} 
	&10 & 2 &  & * & * &  & 3 & 2 &  & 1306 & 3715 &  & 738.7 & 664.9 &  & * & * &  & 65 & 74 &  & 38850 & 114433 &  & 724.4 & 629.3 \\ 
&10 & 5 &  & * & * &  & 5 & 2 &  & 1433 & 3026 &  & 717.9 & 647.1 &  & * & * &  & 81 & 82 &  & 42675 & 130112 &  & 705.1 & 607.8 \\ 
&10 & 10 &  & * & * &  & 5 & 2 &  & 1523 & 2564 &  & 712.5 & 641.1 &  & * & * &  & 74 & 100 &  & 35576 & 174085 &  & 699.8 & 600.3 \\ 
&20 & 2 &  & * & * &  & 69 & 51 &  & 46513 & 76261 &  & 1474.2 & 1325.8 &  & \textbf{.195} & .275 &  & 1000 & 1000 &  & 101509 & 238765 &  & 1404.9 & 1144.5 \\ 
&20 & 5 &  & * & * &  & 103 & 156 &  & 65951 & 209595 &  & 1438.2 & 1274.2 &  & \textbf{.211} & .308 &  & 1000 & 1000 &  & 97940 & 225050 &  & 1375.3 & 1080.9 \\ 
&20 & 10 &  & * & * &  & 215 & 207 &  & 150250 & 349335 &  & 1427.7 & 1256.6 &  & \textbf{.230} & .310 &  & 1000 & 1000 &  & 90864 & 257998 &  & 1366.3 & 1058.2 \\ 
&30 & 2 &  & \textbf{.010}$^{\revised{8}}$ & .011$^{\revised{8}}$ &  & $378$ & $527$ &  & 121358 & 514979 &  & 2230.1 & 2037.7 &  & \textbf{.298} & .441 &  & 1500 & 1500 &  & 38474 & 64240 &  & 2024.0 & 1569.7 \\ 
&30 & 5 &  & \textbf{.011}$^{\revised{8}}$ & .014$^{\revised{8}}$ &  & $571$ & $620$ &  & 164852 & 527847 &  & 2173.9 & 1950.3 &  & \textbf{.336} & .474 &  & 1501 & 1500 &  & 33120 & 64339 &  & 1969.4 & 1448.4 \\ 
&30 & 10 &  & .024$^{\revised{8}}$ & \textbf{.014}$^{\revised{8}}$ &  & $630$ & $638$ &  & 202635 & 585234 &  & 2156.5 & 1919.6 &  & \textbf{.349} & .480 &  & 1500 & 1500 &  & 30579 & 77100 &  & 1951.2 & 1404.0 \\ 
&40 & 2 &  & .042$^{\revised{6}}$ & \textbf{.037}$^{\revised{4}}$ &  & $1551$ & $1615$ &  & 664496 & 2565247 &  & 2979.3 & 2748.6 &  & \textbf{.436} & .545 &  & 2000 & 2000 &  & 23083 & 49050 &  & 2582.0 & 1946.3 \\ 
&40 & 5 &  & \textbf{.045}$^{\revised{6}}$ & .047$^{\revised{2}}$ &  & $1643$ & $1634$ &  & 638323 & 1347868 &  & 2895.6 & 2635.0 &  & \textbf{.579} & .580 &  & 2000 & 2000 &  & 12076 & 30858 &  & 2488.0 & 1751.7 \\ 
&40 & 10 &  & \textbf{.056}$^{\revised{4}}$ & .057$^{\revised{2}}$ &  & $1639$ & $1632$ &  & 599281 & 1584187 &  & 2869.2 & 2595.6 &  & \textbf{.585} & .594 &  & 2000 & 2000 &  & 11847 & 30222 &  & 2456.1 & 1679.6 \\ 
				\Xhline{2\arrayrulewidth}
			\end{tabular}
	\end{adjustbox}}
	\vskip 2ex
	\label{Table: M}
\end{table}

Finally, we study the influence of the big-$M$ parameter. Instead of a coefficient $\gamma=2$ in \cite{park2017bayesian}, we experiment with $M = \gamma \underset{(j,k) \in \overrightarrow{E}}{\max} \, |\beta^{R}_{jk}|$ for $\gamma \in \{2, 5, 10\}$ in Table \ref{Table: M}, where $|\beta^{R}_{jk}|$ denotes the optimal solution of each optimization problem without the constraints to remove cycles. The larger the big-$M$ parameter, the worse the effectiveness of both models. \revised{However, comparing the continuous relaxation objective function values, we observe that} {MISOCP+LN} tightens the formulation using the conic constraints whereas {MIQP+LN} does not have any means to tighten the formulation instead of big-$M$ constraints which have poor relaxation. \revised{In most cases, the {MISOCP+LN} formulation allows more instances to be solved to optimality  than {MIQP+LN}.}   \revised{For larger $m$, because Gurobi solves larger SOCP relaxations in each branch-and-bound node, the MISOCP+LN formulation explores much fewer branch-and-bound nodes and stops with a similar RGAP at termination.} For $M > 2 \underset{(j,k) \in \overrightarrow{E}}{\max} \, |\beta^{R}_{jk}|$, {MISOCP+LN} outperforms {MIQP+LN} in all measures, in most cases.

\subsection{The Effect of Tikhonov Regularization} \label{sec:compl2}
In this subsection, we consider the effect of adding a Tikhonov regularization term to the objective (see Remark \ref{rem:L2}) by considering $\mu \in \{0, \log(n), 2\log(n)\}$ while keeping the values of $\lambda_n = \log(n)$ and $M$ the same as before. Table \ref{Table: mu} demonstrates that for all instances with $\mu>0$, {MISOCP+LN} outperforms {MIQP+LN}. For complete instances with $m=40$ and $\mu=9.2$, {MISOCP+LN} improves the optimality gap from 0.445 to 0.367, an improvement over 21\%. The reason for this improvement is that $\mu>0$ makes the matrix more diagonally dominant; therefore, it makes the conic constraints more effective in tightening the formulation and obtaining a better optimality gap. \revised{Also, {MISOCP+LN} allows more instances to be solved to optimality than {MIQP+LN}.}

\begin{table}[t!]
	\fontsize{60}{30}\selectfont
	\caption{Computational results for different values of $\mu$, * indicates that the problem is solved to the optimality tolerance. Superscript $^{i}$ indicates that out of ten runs, $i$ instances finish before hitting the time limit. Time is averaged over instances that solve within the time limit, RGAP is averaged over instances that reach the time limit. Better RGAPs are in bold.} \label{Table: IP}
	\resizebox{1\textwidth}{!}{
		\begin{adjustbox}{}{}
			\begin{tabular}{llllllllllllllll|lllllllllllllll}							\\ 	 	\Xhline{2\arrayrulewidth}  \Xhline{2\arrayrulewidth} 
				&&& \multicolumn{11}{c}{Moral} &&& \multicolumn{11}{c}{Complete} \\ 
				\Xhline{2\arrayrulewidth}  \Xhline{2\arrayrulewidth} 
				& \multicolumn{2}{c}{Instances} & & \multicolumn{2}{c}{RGAP}  & & \multicolumn{2}{c}{Time} & & \multicolumn{2}{c}{$\#$ nodes} & & \multicolumn{2}{c}{Relaxation OFV} & & \multicolumn{2}{c}{RGAP}  & & \multicolumn{2}{c}{Time} & & \multicolumn{2}{c}{$\#$ nodes} & & \multicolumn{2}{c}{Relaxation OFV} \\ 
				& $m$ &$\mu$  &  & MISOCP & MIQP &  & MISOCP & MIQP & & MISOCP & MIQP &  & MISOCP & MIQP  & &  MISOCP & MIQP &  & MISOCP & MIQP & & MISOCP & MIQP &  & MISOCP & MIQP  \\
				\Xhline{2\arrayrulewidth} 
&10 & 0 &  & * & * &  & 3 & 2 &  & 1306 & 3715 &  & 738.7 & 664.9 &  & * & * &  & 65 & 74 &  & 38850 & 114433 &  & 724.4 & 629.3 \\ 
&10 & 4.6 &  & * & * &  & 4 & 2 &  & 1043 & 2758 &  & 802.0 & 708.5 &  & * & * &  & 69 & 72 &  & 38778 & 119825 &  & 789.3 & 675.7 \\ 
&10 & 9.2 &  & * & * &  & 4 & 2 &  & 1067 & 2231 &  & 858.0 & 748.1 &  & * & * &  & 72 & 74 &  & 36326 & 114383 &  & 843.2 & 712.3 \\ 
&20 & 0 &  & * & * &  & 69 & 51 &  & 46513 & 76261 &  & 1474.2 & 1325.8 &  & \textbf{.195} & .275 &  & 1000 & 1000 &  & 101509 & 238765 &  & 1404.9 & 1144.5 \\ 
&20 & 4.6 &  & * & * &  & 45 & 45 &  & 15111 & 55302 &  & 1604.1 & 1426.5 &  & \textbf{.167} & .242 &  & 1000 & 1000 &  & 102467 & 249490 &  & 1551.7 & 1267.1 \\ 
&20 & 9.2 &  & * & * &  & 43 & 55 &  & 15384 & 62297 &  & 1716.8 & 1515.7 &  &\textbf{.142} & .223 &  & 1000 & 1000 &  & 94360 & 258194 &  & 1668.3 & 1355.1 \\ 
&30 & 0 &  & \textbf{.010}$^{\revised{8}}$ & .011$^{\revised{8}}$ &  & $378$ & $527$ &  & 121358 & 514979 &  & 2230.1 & 2037.7 &  & \textbf{.298} & .441 &  & 1500 & 1500 &  & 38474 & 64240 &  & 2024.0 & 1569.7 \\ 
&30 & 4.6 &  & \textbf{.008}$^{\revised{9}}$ & .011$^{\revised{8}}$ &  & $310$ & $392$ &  & 76668 & 358544 &  & 2432.5 & 2187.7 &  & \textbf{.237} & .387 &  & 1500 & 1500 &  & 45473 & 69258 &  & 2286.4 & 1788.5 \\ 
&30 & 9.2 &  & \textbf{.009}$^{\revised{9}}$ & .010$^{\revised{8}}$ &  & $67$ & $377$ &  & 12410 & 320632 &  & 2612.6 & 2311.4 &  & \textbf{.209} & .367 &  & 1500 & 1500 &  & 41241 & 68661 &  & 2484.3 & 1915.7 \\ 
&40 & 0 &  & .042$^{\revised{6}}$ & \textbf{.037}$^{\revised{4}}$ &  & $1551$ & $1615$ &  & 664496 & 2565247 &  & 2979.3 & 2748.6 &  & \textbf{.436} & .545 &  & 2000 & 2000 &  & 23083 & 49050 &  & 2582.0 & 1946.3 \\ 
&40 & 4.6 &  & \textbf{.027}$^{\revised{8}}$ & .029$^{\revised{4}}$ &  & $1331$ & $1620$ &  & 422654 & 1303301 &  & 3281.6 & 2972.8 &  & \textbf{.354} & .471 &  & 2000 & 2000 &  & 13209 & 30995 &  & 2985.4 & 2261.3 \\ 
&40 & 9.2 &  & \textbf{.020}$^{\revised{8}}$ & .028$^{\revised{6}}$ &  & $870$ & $1507$ &  & 239214 & 1762210 &  & 3575.4 & 3165.3 &  & \textbf{.367} & .445 &  & 2000 & 2000 &  & 13884 & 54638 &  & 3321.7 & 2468.7 \\ 
				\Xhline{2\arrayrulewidth}
			\end{tabular}
		\end{adjustbox}}
		\vskip 2ex
		\label{Table: mu}
	\end{table}

\subsection{Practical Implications of  Early Stopping}\label{sec:compearly}

In this subsection, we evaluate the quality of the estimated DAGs obtained from {MISOCP+LN} by comparing them with the ground truth DAG. To this end, we use \revised{three measures}: the average structural Hamming distance $(\mathrm{SHD})$ which counts the number of arc differences (additions, deletions, or reversals) required to transform the estimated DAG to the true DAG, \revised{the average false positive rate (FPR) which is the proportion of edges appearing in the estimated DAG but not the true DAG and the average true positive rate (TPR) which is the proportion of edges appearing in both the true DAG and the estimated DAG.} Since Gurobi sets a minimum relative gap RGAP$=1e^{-4}$, the solution obtained within this relative gap is considered optimal. Finally, because the convergence of the branch-and-bound process may be slow in some cases, we set a time limit of 100$m$.

To test the quality of the solution obtained with an early stopping criterion, we  set the absolute optimality gap parameter  as $GAP=\frac{\log(m)}{n}s_m$ and the $\ell_0$ regularization parameter as $\lambda_n=\log m$ as suggested by the discussion following Proposition~\ref{EarlyProp} for achieving a consistent estimate. 
We compare the resulting suboptimal solution to the  solution obtained by setting $\mathrm{GAP}= \mathrm{UB} - \mathrm{LB}=0$  to obtain the truly optimal solution.   

Table \ref{Early} shows the numerical results for the average solution time (in seconds) for instances that are solved within the time limit, the number of instances that were not solved within the time limit, the actual absolute optimality gap at termination, \revised{the average FPR, the average TPR,} the average SHD of the resulting DAGs, 
across 10 runs for moral instances. Table~\ref{Early} indicates that the average SHD for  $GAP=\frac{\log(m)}{n}s_m$ is close to that of  the truly optimal solution, \revised{and the  average (FPR) and (TPR) are  the same between setting $\mathrm{GAP}=\frac{\log(m)}{n}s_m$ and $\mathrm{GAP}= 0$ except for $m = 10$}. Note that a lower GAP \revised{generally leads} to a better SHD score. From a computational standpoint, we  observe that by using the early stopping criterion, we are able to obtain consistent solutions \revised{faster}.  In particular, \revised{the average solution time reduces by $25\%$ for $m = 30$. The number of instances which are solved before hitting the $100m$ time limit are the same for $GAP = 0$ and $\mathrm{GAP}=\frac{\log(m)}{n}s_m$.} Furthermore, stopping early  does not  sacrifice too much from the quality of the resulting DAG as can be seen from the SHD scores.

\begin{table}[t!]
	\centering{
		\caption{\revised{Structural Hamming distances (SHD), False Positive Rate (FPR) and True Positive Rate (TPR) for early stopping with $n= 100, \lambda_n = \log(m)$, GAP $\leq \tau$ for moral instances. The superscripts $^{i}$ indicate that out of ten runs, $i$ instances finish before hitting the time limit. Time is averaged over instances that solve within the time limit, GAP and RGAP are averaged over instances that terminate early.}}
\footnotesize{
		\begin{tabular}{ll|cccccc|cccccc}
	\hline
			& & \multicolumn{6}{c|}{$\tau= 0$}  & \multicolumn{6}{c}{$\tau=\frac{\log (m)}{n}s_m$}  \\ \hline
			$m$ & $s_m$ & Time & GAP& RGAP & SHD 
			& FPR & TPR & Time & GAP& RGAP & SHD 
			& FPR & TPR\\ \hline
			10 & 19 & $1.28^{10}$  & $0.00$ & 0.000  & 0.75 
			& $0.04$ & $1.00$ & $1.28^{10}$ & $0.06$ & 0.000 & 0.77 
			& $0.02$ & $1.00$ \\ 
			20 & 58 & $6.15^{9}$ & $0.71$ & 0.000 & 1.50 
			&$0.01$ & $1.00$ &  $6.04^{9}$ & $1.44$ & 0.001  & 2.00 
			&  $0.01$ & $1.00$\\
			30 & 109 & $37.40^{7}$ & $13.96$ & 0.004 & 1.67 
			& $0.00$ & $1.00$ & $27.63^{7}$ & $15.80$ & 0.005 & 1.66 
			& $0.00$ & $1.00$\\ 
			40 & 138 & $935^{2}$ & $63.06$ & 0.131 &5.00 
			& $0.01$ & $1.00$ & $640.15^{2}$ & $66.34$ & 0.132 & 5.00 
			&$0.01$ & $1.00$ \\ \hline
		\end{tabular}
		\label{Early}}
}
\end{table}

\revised{
\subsection{Comparison to Other Benchmarks} \label{sec:comp-sota}

In this section, we compare the performance of  MISOCP against the state-of-the-art benchmarks. These experiments are executed on a laptop with a Windows 10 operating system, an Intel Core i7-8750H 2.2-GHz CPU, 8-GB DRAM using Python 3.8 with Gurobi 9.1.1 Optimizer.  

The benchmarks considered in this section include the top-down approach (EqVarDAG-TD) and the high-dimensional top-down approach (EqVarDAG-HD-TD) of \cite{ChenDrtonWang19}, as well as the high-dimensional bottom-up approach (EqVarDAG-HD-BU) of \cite{pmlr-v84-ghoshal18a}. 
By taking advantage of the conditions for identifiability in linear SEM models, these benchmark procedures offer polynomial-time algorithms for learning DAGs by iteratively identifying a source (top-down) or sink (bottom-up) node based on solving a series of covariance selection problems. 

We compare the performance of the methods on twelve publicly available networks from \cite{manzour2019integer} and \href{https://www.bnlearn.com/bnrepository/}{Bayesian Network Repository (bnlearn)}.  The number of nodes in these networks ranges from $m=6$ to $m=70$. We generate data from both identifiable and non-identifiable error distributions. In the case of identifiable distributions (ID), we generate the data by using random arc weights $\beta$  from $\mathcal{U}[-1,-0.1] \cup \mathcal{U}[0.1,1]$ and $n=500$ samples standard normal errors. The data for the non-identifiable (NID) error distributions was generated similarly, but from normal errors with non-equal error variances chosen randomly from $\{0.5, 1, 1.5\}$.

As an input superstructure graph to MISOCP, other than the true moral graphs, we also consider a superstructure estimate based on the empirical  correlation matrix (CorEst). This estimate---which is guaranteed to be a super set of the DAG skeleton under the faithfulness assumption---was obtained by testing whether each correlation coefficient is nonzero at $0.05$ significance level; the p-values were obtained using the Fisher's Z-transformation for correlation coefficients. 
The MISOCP with true and correlation matrix superstructures are denoted as MISOCP-True and MISOCP-CorEst, respectively, in Table~\ref{Table:Benchmark1}. 
A time limit of $50m$ (seconds), $\lambda = 2\log(n)$ and the Gurobi RGAP of 0.01 are imposed across the experiments.

Measures of performance of the benchmark algorithms are summarized in columns EqVarDAG-TD, EqVarDAG-HD-TD, and EqVarDAG-HD-BU of Table~\ref{Table:Benchmark1}. The column Time reports the solution time in seconds. For all datasets, the  true networks can be used to evaluate the quality of the estimated networks. We report SHD, TPR, and FPR for all the estimated networks. Given that the true causal network cannot be recovered in the setting of non-identifiable data (NID), we also report the structural SHD between the undirected skeleton of the true DAG and the corresponding skeleton of estimated network; this is denoted as SHDs in Table~\ref{Table:Benchmark1}. 

We observe that most of the EqVarDAG methods solve the problem within a second. With respect to the quality of the estimation, EqVarDAG-TD provides better performance in SHD compared to EqVarDAG-HD-TD and EqVarDAG-HD-BU. 
The column RGAP reports the relative gap at early termination. The symbol (*) denotes that the problem is solved to the optimality tolerance. Compared with the benchmarks, MISOCP with a CorEst or true superstructure requires longer solution times; however, MISOCP consistently provides high SHD and SHDs scores in every network.  Moreover, MISOCP is able to provide the best estimation among all methods in most of the  networks.

Finally, we highlight that in the non-identifiable datasets (NID), MISOCP clearly outperforms the benchmarks. This is, of course, not surprising, as the benchmark algorithms rely on the identifiability assumption and are not guaranteed to work if this assumption is violated. In contrast, in this case, MISOCP is guaranteed to find a member of the Markov equivalence class.

\begin{sidewaystable}[ht]
	\caption{The comparison between MISOCP and the state-of-the-art EqVarDAG methods of \cite{ChenDrtonWang19} and \cite{pmlr-v84-ghoshal18a}.}
	\label{Table:Benchmark1}
	\begin{center}
		\scalebox{0.50}{
		\begin{tabular}{|p{0.8cm}			 |p{2.2cm}||p{0.9cm}p{0.8cm}p{0.8cm}p{0.9cm}p{0.9cm}||p{0.9cm}p{0.8cm}p{0.8cm}p{0.9cm}p{0.9cm}||p{0.9cm}p{0.8cm}p{0.8cm}p{0.9cm}p{0.9cm}||p{1.22cm}p{0.8cm}p{0.8cm}p{0.9cm}p{0.9cm}p{0.9cm}||p{1.22cm}p{0.8cm}p{0.8cm}p{0.9cm}p{0.9cm}p{0.9cm}|}
				\hline				
				& & \multicolumn{5}{c||}{EqVarDAG-HD-BU} & \multicolumn{5}{c||}{EqVarDAG-HD-TD}& \multicolumn{5}{c||}{EqVarDAG-TD}& \multicolumn{6}{c||}{MISOCP-CorEst} & \multicolumn{6}{c|}{MISOCP-True}   \\	
				
				\cline{3-29}				
				{Data} & {Network($m$)} 
				& {Time} & {SHD} & {SHDs} & {TPR} & {FPR}
				& {Time} & {SHD} & {SHDs} & {TPR} & {FPR}
				& {Time} & {SHD} & {SHDs} & {TPR} & {FPR}
				& {Time} & {RGAP} & {SHD} & {SHDs} & {TPR} & {FPR}
				& {Time} & {RGAP} & {SHD} & {SHDs} & {TPR} & {FPR} 
				\\			
				\hline
				\multirow{12}{4em}{ID}					
				&Dsep(6) &$\le 1$ & 3 & 3 & 1 & .001
				&$\le 1$ & 3 & 3 & 1 & .001 
				&$\le 1$ & 0 & 0 & 1 & 0
				&$\le 1$ & * & 0 & 0 & 1 & 0  
				&$\le 1$ & * & 0 & 0 & 1 & 0\\					
				&Asia(8) &$\le 1$ & 0 & 0 & 1 & 0
				&$\le 1$ & 0 & 0 & 1 & 0
				&$\le 1$ & 0 & 0 & 1 & 0
				&$\le 1$ & * & 0 & 0 & 1 & 0
				&$\le 1$ & * & 0 & 0 & 1 & 0\\	
				&Bowling(9)	&$\le 1$ & 5 & 4 & 1 & .002
				&$\le 1$ & 1 & 1 & 1 & .001
				&$\le 1$ & 0 & 0 & 1 & 0
				&$\le 1$ & * & 0 & 0 & 1 & 0
				&$\le 1$ & * & 0 & 0 & 1 & 0\\		
				&InsSmall(15)	&$\le 1$ & 11 & 8 & 1 & .005
				&$\le 1$ & 2  & 2 & 1 & .001
				&$\le 1$ & 2  & 2 & 1 & .001  
				&$\ge 750$ & .05  & 0 & 0 & 1 & 0
				&4 & * & 0 & 0 & 1 &0\\		
				
				&Rain(14)	&$\le 1$ & 5 & 4 & 1 & .002
				&$\le 1$ & 2 & 2  & 1 & .001
				&$\le 1$ & 0 & 0 & 1 & 0
				&155 & * & 0 & 0 & 1 & 0
				&1 & * & 0 & 0 & 1 & 0\\		
				
				&Cloud(16)	&$\le 1$ & 15 & 12 & 1 & .006
				&$\le 1$ & 6 & 6 & 1 & .003
				&$\le 1$  & 0 & 0 & 1 & 0
				&$\ge 800$ & .101 & 0 & 0 & 1 & 0 
				&$\le 1$& *   & 0 & 0 & 1 & 0 \\	
				
				&Funnel(18)	&$\le1$ & 10 & 9 & 1 & .004
				&$\le1$ & 4 & 4 & 1 & .002
				&$\le1$ & 0 & 0 & 1 & 0
				& 5 & * & 0 & 0 & 1 & 0
				&$\le1$ & * & 0 & 0 & 1 & 0\\
				
				&Galaxy(20)  &$\le1$ & 16 & 13 & 1 & .007
				&$\le1$ & 9 & 9 & 1 & .004
				&$\le1$ & 1 & 1 & 1 & .001
				&$\ge 1000$ & .048 & 0 & 0 & 1 & 0
				&2 & * & 0 & 0 & 1 & 0\\	
				&Insurance(27)	&2	& 27 & 23 & .962 & .011
				&2	& 11 & 11 & 1 &.005
				&$\le1$& 0 & 0 & 1 & 0
				&$\ge 1350$ & .216 & 0 & 0 & 1 & 0
				&176 & * & 0 & 0 & 1 & 0\\		
				&Factors(27) &2	& 25 & 25 & .977 & .01
				&2	& 17 & 17 & .985 & .007							
				&$\le1$ & 3 & 3 & .956 & 0
				&$\ge 1350$ & .157 & 0 & 0 & 1 & 0
				&$\ge 1350$ & .036 & 0 & 0 & 1 & 0\\		
				
				&Hailfinder(56) &10 & 78 & 70 & .985 & .033 
				&11 & 19  & 19 & 1 & .008
				&$\le1$ & 1  & 1 & .985 & 0
				&$\ge 2800$ & .21  & 13 & 12 & .955 & .004
				&$\ge 2800$ & .052 & 0 & 0 & 1 & 0\\
				
				&Hepar2(70)   &17 & 147 & 139 & .959 & .062 
				&19 & 70  & 70  & .968 & .029
				&2  & 1  & 1  & .992 & 0
				&$\ge 3500$ & .371 & 134 & 122 & .878 & .052
				&$\ge 3500$ & .096 & 0 & 0 & 1 & 0\\																											
				\hline					
				\multirow{12}{4em}{NID}
				&Dsep(6) &$\le 1$ & 4 & 3 & 1 & .002 
				&$\le 1$ & 2 & 1 & 1 & .001 
				&$\le 1$ & 1 & 1 & .833 & 0
				&$\le 1$ & * & 1 & 1 & .833 & 0
				&$\le 1$ & * & 1 & 1 & .833 & 0\\
				&Asia(8) &$\le 1$ & 16 & 13 & .875 & .006
				&$\le 1$ & 6  & 4  & .875 & .002
				&$\le 1$ & 6  & 4  & .875 & .002
				&$\le 1$ & * & 4 & 2 & .875 & .001
				&$\le 1$ & * & 6 & 3 & .875 & .002\\
				&Bowling(9)	&$\le 1$ & 7 & 6 & .909 & .003
				&$\le 1$ & 7 & 5 & .909 & .003
				&$\le 1$ & 4 & 2 & .909 & .001
				&$\le 1$ & * & 2 & 1 & .909 & .001
				&$\le 1$ & * & 4 & 2 & .909 & .001\\
				&InsSmall(15)	&$\le 1$ & 24 & 19 & .96 & .01
				&$\le 1$ & 8  & 7  & .96 & .003
				&$\le 1$ & 8  & 7  & .96 & .003
				&$\ge 750$ & .029  & 7  & 5 & .88 & .002
				&12 & * & 2  & 1 & .96 & .001\\
				&Rain(14)	&$\le 1$ & 17 & 12 & .944 & .007
				&$\le 1$ & 10 & 7  & .944 & .004
				&$\le 1$ & 4  & 1  & .944 & .001
				&43 & * & 7  & 4  & .833 & .002
				&2         & *     & 4  & 1  & .944 & .001\\
				&Cloud(16)	&$\le 1$ & 19 & 14 & .895 & .007
				&$\le 1$ & 20 & 14 & .895 & .007
				&$\le 1$ & 12 & 6  & .895 & .004
				&$\ge 800$ & .062 & 8 & 5 & .947 & .003
				&$\le 1$& *     & 8 & 2 & .947 & .003\\
				&Funnel(18)	&$\le1$ & 12 & 11 & .944 & .005
				&$\le1$ & 6 & 6 & .944 & .002
				&$\le1$ & 1 & 1 & .944 & 0
				& 8 & * & 1 & 1 & .944 & 0
				&$\le1$ & * & 1 & 1 & .944 & 0\\
				&Galaxy(20)  &$\le1$ & 36 & 28 & .955 & .015
				&$\le1$ & 27 & 20 & .955 & .011
				&$\le1$ & 17 & 10 & .955 & .007
				&$\ge 1000$ & .02 & 8 & 5 & .909 & .003
				&2 & * & 6 & 3 & .955 &	.002\\
				&Insurance(27)	&2	& 40 & 32 & 1 & .017
				&2	& 23 & 19 & .981 &.009
				&$\le1$& 13 & 10 & .9615 & .005
				&$\ge 1350$ & .191 & 11 & 9 & .923 & .003
				&$\ge 1350$ & .055 & 6 & 4 & .962 & .002\\
				&Factors(27) &2 & 12 & 11 & .971 & .004
				&2 & 32 & 24 & .882 & .010
				&$\le1$ & 32 & 24 & .765 & .007
				&$\ge 1350$ & .111 & 19 & 15 & .853 & .004
				&$\ge 1350$ & .062 & 22 & 16 & .868 & .006\\
				&Hailfinder(56) &8 & 103 & 88 & .97 & .043 
				&8 & 90  & 70 & .97 & .038
				&2 & 59  & 40 & .894 & .022
				&$\ge 2800$ & .156  & 20 & 11 & .985 & .008
				&$\ge 2800$ & .096 & 15 & 8  & .985 & .006\\
				&Hepar2(70)   &14 & 110 & 95 & .9919 & .048 
				&16 & 73  & 61  & .992 & .031
				&2  & 52  & 40  & .862 & .015
				&$\ge 3500$ & .199 & 33 & 23 & .935 & .011
				&$\ge 3500$ & .112 & 36 & 18 & .96 & .014\\							
				\hline								
		\end{tabular}}
	\end{center}
\end{sidewaystable}	

}

\section{Conclusion} \label{Sec: Conclusion}
In this paper, we study the problem of learning an optimal directed acyclic graph (DAG) from continuous observational data, where the causal effect among the random variables is linear. The central problem is a quadratic optimization problem with regularization. We present a mixed-integer second order conic program ({MISOCP}) which entails a tighter relaxation than existing formulations with linear constraints. Our \revised{numerical} results show that  {MISOCP} can successfully improve the lower bound and results in better optimality gap when compared with other formulations based on big-$M$ constraints, especially for dense and large instances. Moreover, we establish an early stopping criterion under which we can terminate branch-and-bound and achieve a solution which is asymptotically optimal.   \revised{In addition, we show that our method outperforms two state-of-the-art algorithms, especially on non-identifiable datasets.}

\section*{Acknowledgments}
We thank the AE and three anonymous reviewers for their detailed comments that improved the paper. 
Simge K\"u\c{c}\"ukyavuz and Linchuan Wei were supported, in part, by ONR
grant N00014-19-1-2321 and NSF grant CIF-2007814. 
Ali Shojaie was supported by NSF grant DMS-1561814 and NIH grant R01GM114029. Hao-Hsiang Wu is supported, in part, by MOST Taiwan grant 109-2222-E-009-005-MY2.

\appendix

\section{Alternative linear encodings of  constraints \eqref{CP-con2}}
There are several ways to ensure that the estimated graph does not contain any cycles. The first approach is to add a constraint for each cycle in the graph, so that at least one arc in this cycle must not exist in $\mathcal G(B)$.\ A \textit{cutting plane} (CP) method is used to solve such a formulation which may require generating an exponential number of constraints. 
In particular, let $\mathcal{C}$ be the set of all possible directed cycles and $\mathcal{C}_A \in \mathcal{C}$ be the set of arcs defining a cycle.  The CP formulation removes cycles by imposing the following constraints for \eqref{CP-con2} 
\begin{equation} \label{CE}
 \textbf{CP} \quad \sum_{(j,k ) \in \, \mathcal{C}_A} g_{jk} \leq |\mathcal{C}_A|-1, \quad  \forall \mathcal{C}_A \in \mathcal{C}.  
\end{equation}
\revised{This formulation has exponentially many constraints.}

Another way to rule out cycles is by imposing constraints such that the nodes follow a topological order \citep{park2017bayesian}. A topological ordering is a unique ordering of the nodes of a graph from 1 to $m$ such that the graph contains an arc $(j,k)$ if node $j$ appears before node $k$ in the order.  We refer to this formulation as \textit{topological ordering} (TO). 
Define decision variables $z_{jk} \in \{0,1\}$ for all $(j,k) \in \overrightarrow{E}$ and  $o_{rs} \in \{0,1\}$ for all $r, s \in \{1, \dots, m\}$. The variable $z_{jk}$ takes value 1 if there is an arc $(j,k)$ in the network, and $o_{rs}$ takes value 1 if the topological order of node $r$ equals $s$.  The TO formulation rules out cycles in the graph by the following constraints
\begin{subequations}\label{TO}
	\begin{alignat}{3}
\textbf{TO}	\quad & 	\label{TO-con3}  g_{jk}  \leq z_{jk}, \quad && \forall (j,k) \in \overrightarrow{E}, \\
	\label{TO-con4} & z_{jk} - m z_{kj} \leq \sum_{s \in V} s \, (o_{ks} - o_{js}), \quad&& \forall (j,k) \in \overrightarrow{E},\\
	\label{TO-con5} & \sum_{s \in V} o_{rs} =1 \quad && \forall r \in V, \\ 
	\label{TO-con6} & \sum_{r \in V} o_{rs} =1 \quad  &&\forall s \in V.
	\end{alignat}
\end{subequations}
\revised{This formulation has $\mathcal{O}(m^2)$ variables and $\mathcal{O}(|\overrightarrow{E}|)$ constraints.}

\bibliographystyle{plainnat}
\bibliography{ref}

\end{document}